%% file: article.tex
\documentclass{amsart}

\input{preamble}
\input{math}

\begin{document}

\title{A model for framed configuration spaces of points}

\author{Ricardo~Campos}
\address{Ricardo~Campos:  Institut de Mathématiques de Toulouse, UMR5219, Université de Toulouse, CNRS, UPS, F-31062 Toulouse Cedex 9, France}
\email {ricardo.campos@math.univ-toulouse.fr}

\author{Julien~Ducoulombier}
\address{Julien~Ducoulombier: Université Sorbonne Paris Nord, LAGA, CNRS, F-93430, Villetaneuse, France}
\email{ducoulombier@math.univ-paris13.fr}

\author{Najib~Idrissi}
\address{Najib~Idrissi: Université Paris Cité, Sorbonne Université, CNRS, IMJ-PRG, F-75013 Paris, France}
\email{najib.idrissi-kaitouni@u-paris.fr}

\author{Thomas~Willwacher}
\address{Thomas~Willwacher: Department of Mathematics, ETH Zurich, Zurich, Switzerland}
\email{thomas.willwacher@math.ethz.ch}


\begin{abstract}
  We study configuration spaces of framed points on oriented closed smooth manifolds.
  Such configuration spaces admit natural actions of the framed little discs operads, that play an important role in the study of embedding spaces of manifolds and in factorization homology.
  We construct real combinatorial models for these operadic modules, for orientable closed smooth manifolds.
\end{abstract}

\maketitle

\section{Introduction}
\label{sec:introduction}

Let $N$ be a smooth oriented manifold
and let $\mathbb D^n$ denote the $n$-dimensional disc.
Consider the space of orientation-preserving embeddings of $n$-discs in $N$,
\[ \Discs_k^n(N) = \Emb(\underbrace{\mathbb D^n \sqcup \dots \sqcup \mathbb D^n}_{k\text{ discs}}, N). \]
The collection of spaces $\Discs^n(N) \coloneqq \{\Discs_k^n(N)\}_{k\geq 0}$ admits a natural action of the framed little $n$-discs operad, $E_n^{\fr} \simeq \Discs^{n}(\mathbb{D}^{n})$.
The purpose of this paper is to compute the real homotopy type of the right $E_n^{\fr}$-module $\Discs^n(N)$, by providing combinatorial (graphical) models when $N$ is oriented and closed.

Our result is one step towards a real version of the Goodwillie--Weiss manifold calculus, as we shall briefly outline.
The Goodwillie--Weiss manifold calculus interprets embedding spaces in terms of the operadic right modules $\Discs^n(N)$. If $M$ and $N$ are smooth oriented manifolds of dimensions $m$ and $n$ respectively such that $n-m\geq 3$, then there is a weak equivalence~\cite{BoavidaWeiss2013,Turchin2013}:
\[ \Emb(M,N) \simeq \Map^h_{E_{m}^{\fr}-\mathrm{mod}} (\Discs^m(M),\Discs^m(N)), \]
where the right-hand side is the derived mapping space of morphisms from $\Discs^{m}(M)$ to $\Discs^{m}(N)$ viewed as modules over the framed little discs operads.

This right-hand side is still hard to compute, but one may hope to determine at least its real or rational homotopy type as follows.
Denote by $\Omega(E_{n}^{\fr})$ a cooperad in differential graded commutative algebras quasi-isomorphic to e.g.\ Sullivan forms on $E_n^{\fr}$, and similarly by $\Omega(\Discs^m(M))$, and $\Omega(\Discs^n(N))$ cooperadic comodules quasi-isomorphic to e.g.\ Sullivan forms on $\Discs^m(M)$ and $\Discs^n(N)$. Then one has a natural map
\begin{multline} \label{equ:ratGW}
  \Map^h_{E_{m}^{\fr}-\mathrm{mod}} (\Discs^m(M),\Discs^m(N)) \\ \to \Map^h_{\Omega(E_{m}^{\fr})-\mathrm{comod}} (\Omega(\Discs^m(N)),\Omega(\Discs^m(M))).
\end{multline}
In good cases, one may hope that the target of the map \eqref{equ:ratGW} can be effectively computed, and that the map is is finite-to-one on $\pi_0$ and a rational equivalence on each connected component (see~\cite{FresseTurchinWillwacher2020}).
Analogous results have been proved for the case of higher dimensional long knots ($M = \R^{m}$, $N = \R^{n}$, $n-m\geq 3$) in \cite{FresseTurchinWillwacher2017} and for more general spaces of embeddings in $N=\R^n$ in \cite{FresseTurchinWillwacher2020}.
The goal of the present paper is to contribute to solving the general case by providing combinatorial models for $\Omega(\Discs^m(M))$, where $M$ is a closed oriented manifold of dimension $m=\dim M$ (i.e.\ tools to effectively compute the right-hand side of~\eqref{equ:ratGW}).

Let us briefly describe our results. For technical reasons we will be working with a configuration space version of $\Discs^m(M)$.
The ordered configuration space of $k$ points on $M$ is given by
\[ \Conf_k(M) \coloneqq \{(x_1,\dots,x_k) \in M^k \mid x_i\neq x_j, \text{ for $i\neq j$} \}. \]
Let us fix a Riemannian metric on $M$.
The configuration space of $k$ framed points on $M$, $\Conf_k^{\fr}(M)$, consists of configurations in $\Conf_k(M)$ with the additional prescription of positively oriented orthonormal bases of the tangent spaces at each of the points in the configuration,
\[
  \Conf_k^{\fr}(M)
  \coloneqq
  \left\{ (x,B_1,\dots,B_k)
  \middle|
  \begin{array}{ll}
    x \in \Conf_{k}(M), \\
    B_i \text{: oriented orthonormal basis of } T_{x_i}M
  \end{array}
  \right\}.
\]
Then one has a natural weak equivalence

\[
  \Discs^m_k(M) \xrightarrow{\sim} \Conf^{\fr}_k(M)
\]
sending a configuration of discs to the configuration of their centers, equipped with the orthonormalization of the push-forward of the standard frame at the center of the discs.
Furthermore we consider the Fulton--MacPherson--Axelrod--Singer compactification $\FM_{M}$ of $\Conf(M)$ and $\FFM_{M}$ of $\Conf^{\fr}(M)$.
The latter object carries a natural action of a Fulton-MacPherson-Axelrod-Singer-version of the framed $E_m$-operad $\FM_m^{\fr}$, see \cite{Markl1999}.

Some of the authors described in \cite{CamposWillwacher2016, Idrissi2016} a graphical real model $\Graphs_M(k)$ for $\Conf_k(M)$.
Elements of $\Graphs_M(k)$ are linear combinations of undirected diagrams with $k$ numbered ``external'' vertices, some further ``internal'' vertices, and zero, one, or more decorations in the cohomology $H(M)$ at each vertex:
\[
  \begin{tikzpicture}[every pin edge/.style={densely dotted}, every pin/.style={font=\footnotesize}, scale=2]
    \node[ext, pin={above left}:{$\omega_2$}] (v1) at (0,0) {$\scriptstyle 1$};
    \node[ext] (v2) at (.5,0) {$\scriptstyle 2$};
    \node[ext] (v3) at (1,0) {$\scriptstyle 3$};
    \node[ext, pin={above right}:{$\omega_3$}] (v4) at (1.5,0) {$\scriptstyle 4$};
    \node[int] (w1) at (.25,.5) {};
    \node[int, pin={above right}:{$\omega_1$}, pin={above left}:{$\omega_1$}] (w2) at (1,1) {};
    \node[int] (w3) at (1,.5) {};
    \draw (v1) edge (v2) edge (w1) (w1) edge (w2) edge (w3) edge (v2) (v3) edge (w3) (v4) edge (w3) edge (w2) (w2) edge (w3);
  \end{tikzpicture}
\]
There is a quasi-isomorphism
\beq{equ:stGmap}
\stG_M(k) \to \OmPA(\FM_M(k))
\eeq
into the (piecewise semi-algebraic) differential forms on $\FM_M$ given by natural ``Feynman rules''.
In this paper we extend $\stG_M$ to a model $\stG_M^{\fr}$ for the framed configuration space.
As a graded algebra, it is merely obtained by further decorating each external vertex by a model of $\SO(m)$; however, the differential is more complex.
Our graph complex comes with a zigzag:
\begin{equation}
  \stG_M^{\fr}(k) \leftarrow \cdot \to  \OmPA(\FFM_M(k)).
\end{equation}
We furthermore describe an explicit cooperadic coaction of a graphical model $\stG^{^\fr}_{m}$ of $E^{\fr}_m$ on $\stG_M^{\fr}$.
Our main result is then that this combinatorial coaction indeed models the desired topological action of $E^{\fr}_m$ on $\FM_M$.
\begin{thm}[See Theorem~\ref{thm:model-ffmm}]
  The zigzag $\stG_M^{\fr}(k) \gets \cdot \to \OmPA(\FFM_M(k))$ is a weak equivalence, and it is compatible with the action of $\FFM_{m}$ on $\FFM_{M}$.
\end{thm}

We furthermore adopt an alternative viewpoint, which serves as an intermediate result in the proof of the theorem above.
In general, the non-framed configuration spaces $\FM_M$ do not carry an action of the little discs or Fulton--MacPherson operads.
However, there exists a fiberwise version $\FM_m^M$ of the Fulton--MacPherson operad.
It is an operad in topological spaces over $M$, and the fiber over $x\in M$ is a compactification of the space of configurations of points in the tangent space $T_xM$.
The collection $\FM_{M}$ is equipped with an action of $\FM_{m}^{M}$, almost tautologically.

Our second main result is to upgrade the dg commutative algebra model $\stG_M$ of $\FM_M$ so as to capture also the action of $\FM_m^M$ on $\FM_M$:
\begin{thm}[See Theorem~\ref{thm:graphs-m-model}]
  There is a graph complex $\stG_{m}^{M}$, which is a model for $\FM_{m}^{M}$, and which acts on the model $\stG_{M}$ of $\FM_{M}$ from~\cite{CamposWillwacher2016}, so that the map \eqref{equ:stGmap} respects the (homotopy) comodule structures on both sides.
\end{thm}

This paper is organized as follows.
In Section~\ref{sec:backgr-recoll}, we lay out the necessary background for the next sections: the Axelrod--Singer--Fulton--MacPherson compactifications $\FM_{m}$ and $\FM_{M}$, the definition of homotopy operads and related structures, Kontsevich's proof of the formality of $\FM_{m}$, the graphical models for $\FM_{M}$, and the equivariant graphical models for $\FM_{m}^{\fr} = \FM_{m} \rtimes \SO(m)$.
In Section~\ref{sec:fibered-little-discs}, we define a fiberwise version $\FM_{m}^{M}$ of the operad $\FM_{m}$, which acts almost tautologically on $\FM_{M}$ whether $M$ is parallelized or not.
We provide a graphical model for this operad.
In Section~\ref{sec:fram-conf-module}, we define the framed configuration module $\FM_{M}^{\fr}$ as a general instance of a ``framing construction'' for right multimodules.
We define a graphical model for $\FM_{M}^{\fr}$ and we prove that it is compatible with the graphical model for $\FM_{m}^{\fr}$.
Finally, in Section~\ref{sec:frame-change}, we explain how our graphical model is related to the models of~\cite{CamposWillwacher2016,Idrissi2016} when the manifold $M$ happens to be parallelized.

This article contains three appendices
In the first, Appendix~\ref{sec:explicit_prop}, we write down an explicit formula for the ``equivariant propagator'', a certain $\SO(n)$-equivariant form on $\FM_{m}$ that is used to define the graphical model for $\FM_{m}^{\fr}$.
Then, in Appendix~\ref{sec:homot-comm-diagr}, we prove a key part of the main theorem of Section~\ref{sec:zigzag}, namely, that a certain diagram is commutative.
Finally, in Appendix~\ref{sec:hopf-formality}, we prove a result regarding Hopf formality of compact Lie groups.

\begin{rem}
  In this paper we will deviate notationally from \cite{CamposWillwacher2016} and denote the graphical commutative algebra model of the configuration space by $\stG_M$ instead of ${}^*\stG_M$ to simplify the notation. We hope that nevertheless no confusion arises.
\end{rem}

\begin{rem}
  We only provide models for $\Discs^{m}(M)$ in Equation~\eqref{equ:ratGW}.
  The spaces $\Discs^{m}(N)$ (for $m < \dim N$) are homotopy equivalent to $m$-framed configuration spaces, i.e.\ configurations of points endowed with $m$ linearly independent sections of the tangent bundle~\cite[Section~2]{Turchin2013}, and are more complicated.
\end{rem}

\subsection{Acknowledgments}

R.C.\ was supported by the  Swiss National Science Foundation  Early Postdoc.Mobility grant \texttt{P2EZP2\_174718}.
J.D., N.I. and T.W.\ were supported by the ERC starting grant 678156--GRAPHCPX during the completion of this work.
N.I.\ contributes to the IdEx University of Paris ANR-18-IDEX-0001.
The authors are very thankful to the referee for an extremely thorough review of the paper.

\section{Background and recollections}
\label{sec:backgr-recoll}

From now on, we use the letter $n$ for the dimension of the manifold $M$ under consideration, as $m$ will be reserved for Maurer--Cartan elements.

\subsection{The Fulton-MacPherson-Axelrod-Singer compactifications}
\label{sec:conf-spac-fult}

The configuration spaces of a manifold, even a compact one, are generally not compact.
One way to fix this is through the Fulton--MacPherson--Axelrod--Singer compactification process.
We will only give a quick account and refer to~\cite{FultonMacPherson1994,AxelrodSinger1994,Sinha2004,LambrechtsVolic2014} for more details.

First, let us consider $M = \R^{n}$.
We can first mod out the translations and the positive rescaling in $\Conf_{k}(\R^{n})$ to obtain the space $\Conf_{k}(\R^{n}) / (\R^{n} \rtimes \R_{>0})$, which is a manifold of dimension $nk-n-1$ if $k \ge 2$ (and a singleton otherwise).
The Fulton--MacPherson compactification $\FM_{n}(k)$ is a manifold with corners whose interior is $\Conf_{k}(\R^{n}) / \R^{n} \rtimes \R_{> 0}$, and the inclusion is a homotopy equivalence.
The elements of $\FM_{n}(k)$ can be seen as configurations of $k$ points in $\R^{n}$, where the points are allowed to become ``infinitesimally close'' to each other.
The collection $\FM_{n} = \{ \FM_{n}(k) \}_{k \ge 0}$ of all these spaces assembles to form a topological operad, the Fulton--MacPherson operad, obtained by considering ``insertion'' of infinitesimal configurations. The element obtained from the operadic composition in the picture below can be interpreted as follows: the points $1$, $5$ and $6$ are infinitesimally close, and so are the points $2$, $3$ and $4$. Moreover, the distance between the points $2$ and $4$ is infinitesimally small compared to the distance between $2$ and $3$.

\begin{figure}[!h]
  \begin{center}
    \includegraphics[width=\textwidth]{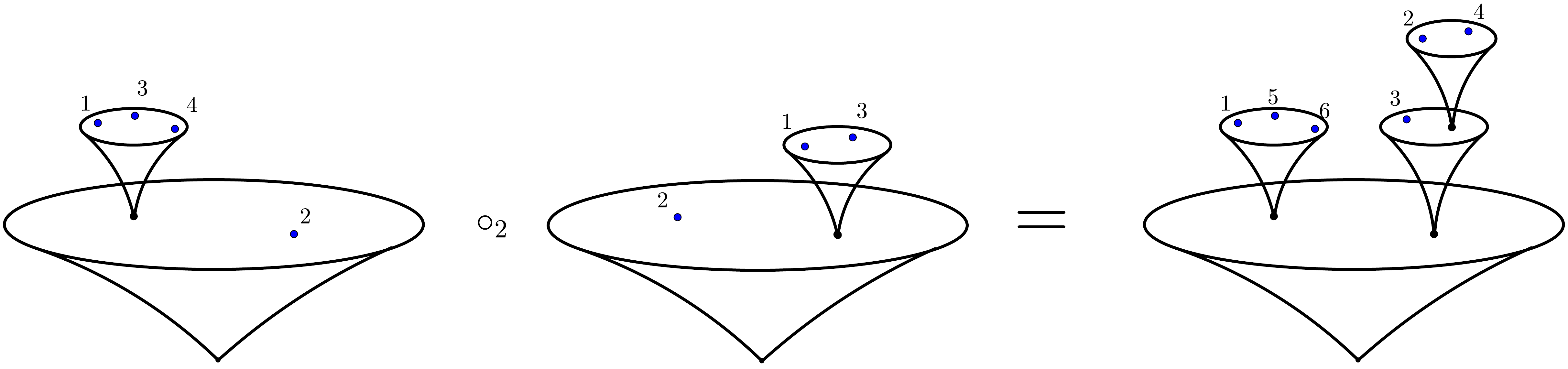}
    \caption{Illustration of the operadic structure of $\FM_{n}$.}
  \end{center}
\end{figure}

\begin{rem}
  The operad $\FM_{n}$ is weakly equivalent to the better-known little discs operad, i.e.\ it is an $E_{n}$-operad, see~\cite[Proposition~3.9]{Salvatore2001}.
\end{rem}

Let us now consider the case of $M$ being a closed $n$-manifold.
The compactification $\FM_{M}(k)$ is again a manifold with corners, with interior $\Conf_{k}(M)$ (with no quotient), and the inclusion is a homotopy equivalence.
Elements of $\FM_{M}(k)$ can also be seen as configurations of $k$ points in $M$ where points can become infinitesimally close to each other.
When they do become infinitesimally close, we see them as defining an infinitesimal configuration in the tangent space of $M$ at their location.
If $M$ is framed, i.e.\ if we can coherently identify the tangent space at every point of $M$ with $\R^{n}$, then we can insert an infinitesimal configuration from $\FM_{n}$ into a configuration of $\FM_{M}$ and thus obtain the structure of a right operadic $\FM_{n}$-module on $\FM_{M}$.

\begin{figure}[!h]
  \centering
  \includegraphics[width=\textwidth]{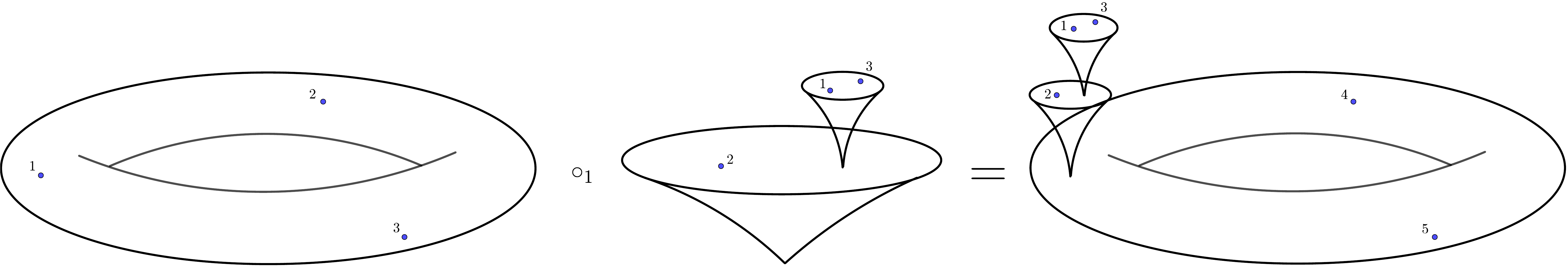}
  \caption{Illustration of the right $\FM_{n}$-module structure on $\FM_{M}$.}
\end{figure}

We can restrict the canonical projections $p_{i} : M^{k} \to M$ (for $1 \le i \le k$) to $\Conf_{k}(M)$, and then extend them to the compactification:
\begin{align}
  \label{equ:projections}
  p_{i} : \FM_{M}(k) & \to M, & 1 \le i \le k.
\end{align}

\subsection{Homotopy (co)operads and (co)modules, rational homotopy theory of operads} \label{sec:htpyoperads}
A basic technical problem in the rational (or real) homotopy theory of operads is that for a topological operad $\op T$ the (PL or smooth, if defined) differential forms $\Omega(\op T)$ do not form a cooperad.

This is due to the functor $\Omega : \Top \to \dgca^{\mathrm{op}}$ being oplax monoidal, but not lax monoidal, so that the cocomposition maps are encoded by a zigzag
\[
  \Omega(\op T(k+l-1))\to \Omega(\op T(k) \times \op T(l)) \xleftarrow{\sim} \Omega(\op T(k))\otimes \Omega(\op T(l)).
\]
However, there is no natural direct map from the left to the right as would be required for a cooperad.
One can use one of three workarounds for this problem: (i) use homotopy operads as in \cite{LambrechtsVolic2014,KhoroshkinWillwacher2017}; (ii) alter the functor $\Omega$ as in \cite{Fresse2017a,Fresse2017b}; or (iii) work with topological vector spaces and the projectively completed tensor product so that the right-hand arrow above becomes an isomorphism.
We will follow here the first approach, using an ``ad hoc'' notion of homotopy operad proposed in \cite{LambrechtsVolic2014}, see \cite{KhoroshkinWillwacher2017} for more details.

\begin{rem}\label{rem:htpy-forests}
  Although we expect that the homotopy theory of the notion of homotopy operads from~\cite{LambrechtsVolic2014,KhoroshkinWillwacher2017} can be worked out to be equivalent to the rational homotopy theory of~\cite{Fresse2017a,Fresse2017b}, this has not yet been done.
\end{rem}

Concretely, let $\fT$ be the category whose objects are forests of rooted trees from~\cite{KhoroshkinWillwacher2017}, and whose morphisms are generated by contracting edges of trees, by cutting edges and by tree isomorphisms.
The category $\fT$ is symmetric monoidal, the monoidal product being the disjoint union of forests.

\begin{align*}
  \begin{tikzpicture}
    \node at (0,0) {
      \begin{tikzpicture}
        \coordinate (v) at (0,0.5);
        \coordinate (w) at (0,-0.5);
        \draw (v) edge (w) edge +(-.5,-.5)  edge +(.5,-.5)  edge +(0,.5)
        (w)   edge +(-.5,-.5)  edge +(.5,-.5)  edge +(0,-.5);
      \end{tikzpicture}
    };
    \node at (1,-0.25) {$\to$};
    \node at (2,0) {
      \begin{tikzpicture}
        \coordinate (v) at (0,0);
        \draw (v)  edge +(-.5,-.5)  edge +(.5,-.5)  edge +(0,.5) edge +(-.25,-.5)  edge +(.25,-.5)  edge +(0,-.5);
      \end{tikzpicture}
    };
    \node at (5,0) {
      \begin{tikzpicture}
        \coordinate (v) at (0,0.5);
        \coordinate (w) at (0,-0.5);
        \draw (v) edge (w) edge +(-.5,-.5)  edge +(.5,-.5)  edge +(0,.5)
        (w)   edge +(-.5,-.5)  edge +(.5,-.5)  edge +(0,-.5);
      \end{tikzpicture}
    };
    \node at (6.2,-0.5) {$\longrightarrow$};
    \node at (6.15,0.3) {\resizebox{0.5cm}{!} {
        \begin{tikzpicture}
          \coordinate (v) at (0,0.5);
          \coordinate (w) at (0,-0.5);
          \draw (v) edge[dashed] (w) edge +(-.5,-.5)  edge +(.5,-.5)  edge +(0,.5)
          (w)   edge +(-.5,-.5)  edge +(.5,-.5)  edge +(0,-.5);
        \end{tikzpicture}}
    };
    \node at (7.3,-.2) {
      \begin{tikzpicture}
        \coordinate (v) at (0,0.5);
        \draw (v) edge +(0,-.5) edge +(-.5,-.5)  edge +(.5,-.5)  edge +(0,.5);
      \end{tikzpicture}
    };
    \node at (8.8,-.2) {
      \begin{tikzpicture}
        \coordinate (v) at (0,0.5);
        \draw (v) edge +(0,-.5) edge +(-.5,-.5)  edge +(.5,-.5)  edge +(0,.5);
      \end{tikzpicture}
    };
  \end{tikzpicture}
\end{align*}

We say that a (nonunital) homotopy operad in a symmetric monoidal category $\catC$ with weak equivalences is a symmetric monoidal functor
\[
  \op P : \fT \to \catC,
\]
such that all cutting morphisms are sent to weak equivalences.
Concretely, a (nonunital) homotopy operad consists of the data
\begin{itemize}
  \item For every tree $T$ an object $\op P(T)$, on which the automorphisms of $T$ act.
  \item For every edge contraction $T\to T'$ (a map of trees) we have a corresponding map $\op P(T)\to \op P(T')$.
  \item If $T$ is obtained by grafting $T_1$ and $T_2$ then we have a weak equivalence $\op P(T)\xrightarrow{\sim} \op P(T_1)\otimes \op P(T_2)$.
\end{itemize}
These data must satisfy natural compatibility conditions.
Any ordinary (nonunital) operad $\op P$ is also a homotopy operad, by setting
\[
  \op P(T) \coloneqq \otimes_T \op P \coloneqq \bigotimes_{v\in \text{vertices of }T} \op P (|v|)
\]
to be the tree-wise tensor product, the ``contraction'' morphism to agree with the operadic composition and the ``cutting'' maps are the isomorphisms
\[
  \op P(T) \cong \op P(T_1)\otimes \op P(T_2).
\]
There is also a variant for unital operads. One may define a category $\fT_1$ (see~\cite{KhoroshkinWillwacher2017}) having the same objects as $\fT$, but where in addition we have a generating morphism of creating a vertex having a single input and a single output on any edge.
In particular, there is a morphism from the empty tree to the $1$-corolla.
A unital homotopy operad is then a symmetric monoidal functor $\fT_1\to \catC$.

Dually we define a homotopy cooperad in $\op D$ as a contravariant symmetric monoidal functor
\[
  \op D : \fT \to \catC^{op}.
\]
The main example is as follows: Suppose $\op T$ is a topological operad. Then the (PL) forms $\Omega(\op T)$ form a homotopy cooperad in the category $\dgca$ of dg commutative algebras.
We will call such objects homotopy Hopf cooperads for short.
The corresponding functor
\[
  \Omega(\op T) : \fT \to \dgca
\]
is defined such that
\begin{equation}\label{eq:omegaishomotopycoop}
  \Omega(\op T) : T \mapsto \Omega(\times_T \op T).
\end{equation}

The contraction morphisms are the pullbacks of composition morphisms in $\op T$ and the ``cutting'' morphisms are the natural maps
\[
  \Omega(T_1)\otimes \Omega(T_2) \xrightarrow{\sim} \Omega(T).
\]

We will also work with the corresponding notion of homotopy operadic right modules.
Let $\fTm$ be a category whose objects are forests with one marked tree.
The morphisms are generated by edge contractions and edge cuts. Cutting an edge in the marked tree will leave the upper (closer to the root) subtree marked, and the other subtree unmarked.
The category $\fTm$ is naturally a monoidal category module over $\fT$.
Now suppose that
\[
  \op P: \fT\to \catC
\]
is a homotopy operad in the symmetric monoidal category $\catC$ (i.e., a symmetric monoidal functor), then a homotopy right operadic $\op P$-module $\op M$ is a functor
\[
  \op M : \fTm\to \catC
\]
so that the pair $(\op P, \op M)$ respects the given structure, and such that all cutting morphisms are sent to weak equivalences.
More precisely, $\op M$ is specified by the following data
\begin{enumerate}
  \item A collection of objects $\op M(T)$ for every (marked) tree $\op T$.
  \item Contraction morphisms $\op M(T)\to \op M(T')$.
  \item Cutting morphisms (weak equivalences) $\op M(T) \to \op M(T_1)\otimes \op P(T_2)$.
\end{enumerate}
Every operadic right module $\op M$ over an operad $\op P$ is in particular a homotopy operadic right module.

Dually, we define the notion of homotopy cooperadic right comodule.
In particular we will consider homotopy cooperadic right comodules in the category $\dgca$, which we call Hopf right comodules.
The main example will be as follows.
Let $\op T$ be again a topological operad, and $\op M$ a topological operadic right $\op T$-module.
Then the (PL) forms $\Omega(\op T)$ form a homotopy Hopf cooperad, as we saw. Furthermore the forms $\Omega(\op M)$, defined such that for a (marked tree)
\beq{equ:examplemoduletree}
T=
\begin{tikzpicture}[baseline=-4ex, scale=.5]
  \draw node {}
  child {
      child { node {$T_1$} }
      child { node {$\cdots$} }
      child { node {$T_r$} }
    };
\end{tikzpicture}
\eeq
we have
\[
  \Omega(\op M)(T) := \Omega\left( \op M(r) \times (\times_{T_1}\op T) \cdots (\times_{T_r} \op T )\right),
\]
naturally form a homotopy Hopf right comodule for $\Omega(\op T)$.

We will not fully develop the homotopy theory of homotopy (Hopf) (co)modules here.
We just say that we equip the category of homotopy right comodules with a structure of a homotopical (or $\infty$-)category by declaring the weak equivalences to be the morphisms that are object-wise weak equivalences.

For us, understanding the ``naive'' homotopy type of the topological operad $\op T$ acting on the topological operadic right module $\op M$ shall mean understanding the weak equivalence class (quasi-isomorphism type) of the pair consisting of the homotopy Hopf cooperad $\Omega(\op T)$ and its homotopy Hopf comodule $\Omega(\op M)$. For us a \emph{model} of $(\op T, \op M)$ shall be a pair consisting of a homotopy Hopf cooperad and a homotopy right comodule, such that the pair can be connected to $(\Omega(\op T),\Omega(\op M))$ by a zigzag of quasi-isomorphisms.

We finally remark that a ``proper'' rational homotopy theory of topological operads has been developed by B. Fresse \cite{Fresse2017a,Fresse2017b, Fresse2018}.
Concretely, he constructs a model category structure on (ordinary) Hopf cooperads, together with a Quillen adjunction with the category of topological operads.
Furthermore, he shows that morphisms of homotopy Hopf cooperads in our sense may be lifted to morphisms of ordinary dg Hopf cooperads in his framework, thus embedding our computations in a more satisfying homotopy theoretical framework.

\begin{rem}
  We want to emphasize that the notion ``homotopy operad'' is a bit of a misnomer, since homotopy operads are not objects in the homotopy category of operads. ``Lax operad'' could be a better name.
\end{rem}

\subsection{Formality of $\FM_{n}$}
\label{sec:graphical-models}

The little discs operads are known to be formal over $\Q$, i.e.\ their rational cohomology completely determines their rational homotopy type as operads~\cite{Kontsevich1999,Tamarkin2003,LambrechtsVolic2014,Petersen2014,FresseWillwacher2015}.
There are several methods to prove this result.
Here we recall the one pioneered by Kontsevich (which works over $\R$), based on graphical models, and that was recently applied to closed manifolds~\cite{CamposWillwacher2016,Idrissi2016} (see also Section~\ref{sec:graph-models-fm_m}) and compact manifolds with boundary~\cite{CamposIdrissiLambrechtsWillwacher2018} by some of the authors and Lambrechts.

For a topological operad $P$ of finite cohomological type, its cohomology $H(P)$
(e.g.\ over $\R$) is naturally a Hopf cooperad, i.e.\ a cooperad in the category of commutative differential graded algebras (here, with a trivial differential).
The forms on $P$ (for a suitable notion of ``forms'') $\Omega(P)$ are a homotopy Hopf cooperad. The formality of $\FM_{n}$ is then the statement that $H(\FM_{n})$ and $\Omega(\FM_{n})$ are quasi-isomorphic, i.e., they can be connected by a zigzag of quasi-isomorphisms of homotopy Hopf cooperads.

To set the notation, recall that the cohomology of $\FM_{n}(k) \simeq \Conf_{k}(\R^{n})$ is given by the following algebra, with generators $\omega_{ij}$ of degree $n-1$:
\begin{equation}
  \label{eq:5}
  H(\FM_{n}(k)) = S(\omega_{ij})_{1 \le i, j \le k} / \bigl( \omega_{ii}, \omega_{ij}^{2}, \omega_{ij} \omega_{jk} + \omega_{jk} \omega_{ki} + \omega_{ki} \omega_{ij}, \omega_{ji} - (-1)^n \omega_{ij} \bigr).
\end{equation}

Kontsevich \cite{Kontsevich1999} built a Hopf cooperad $\stG_{n}$ to connect $H(\FM_{n})$ with the forms on $\FM_{n}$ as follows.
Elements of $\stG_{n}(k)$ are linear combinations of graphs with two types of vertices: ``external'' vertices, numbered from $1$ to $k$, and an arbitrary number of ``internal'' vertices, undistinguishable and usually drawn in black.
The edges are formally directed, but an edge is identified with $(-1)^{n}$ times its opposite edge, so we will usually not draw the orientation.
The total degree of a graph is $(n-1)$ times the number of edges, minus $n$ times the number of internal vertices.
We mod out by graphs containing connected components with only internal vertices.
The differential of a graph is obtained as a sum over all possible ways of contracting an edge connected to an internal vertex.
The product glues graphs along external vertices, while the cooperad structure maps collapse subgraphs.
In order to have a well-defined coaction on our graph complex, we mod out by all graphs containing internal vertices of valence $\le 1$.

\begin{rem}\label{rem:tadpoles}
  To be consistent with what follows, note that we explicitly allow ``tadpoles'', i.e.\ edges between a vertex and itself, as well as multiple edges.
  This does not change the quasi-isomorphism type of the graph complex, see for instance the argument of \cite[Proposition 3.4]{Willwacher2014}.
  Mind however that for $n$ even graphs with double edges are automatically zero by symmetry, since the symmetry of switching the two edges is odd.
  Similarly, for $n$ odd graphs with tadpoles are zero since they have an odd symmetry by flipping the tadpole.
\end{rem}

\begin{figure}[!h]
  \centering
  \begin{tikzpicture}
    \node[ext] (v1) at (0,0) {$\scriptstyle 1$};
    \node[ext] (v2) at (1,0) {$\scriptstyle 2$};
    \node[ext] (v3) at (2,0) {$\scriptstyle 3$};
    \node[ext] (v4) at (3,0) {$\scriptstyle 4$};
    \node[int] (i1) at (0.5,1) {};
    \node[int] (i2) at (1.5,1) {};
    \node[int] (i3) at (2.5,1) {};
    \draw (i1) edge (i2) edge[bend left] (i3) edge (v1) edge (v2)
    (i2) edge (v2) edge (v3)
    (i3) edge (v3) edge (v4)
    (v1) edge (v2) edge[bend right] (v3);
  \end{tikzpicture}
  \caption{Illustration of an element in $\stG_{n}(4)$ with $3$ internal vertices.}
\end{figure}
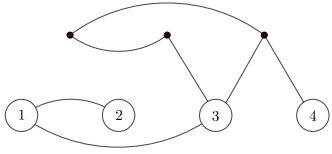

One may then define a first morphism $\stG_{n} \to H(\FM_{n})$ by sending an edge between $i$ and $j$ to $\omega_{ij}$ (in particular graphs with tadpoles are mapped to $0$), and any graph with internal vertices to zero.
For technical reasons, we need to work with piecewise semi-algebraic (PA) forms $\OmPA(\FM_{n})$, cf. \cite{HardtLambrechtsTurchinVolic2011}.
The second morphism $\omega : \stG_{n} \to \OmPA(\FM_{n})$ is defined using configuration space integrals.
First note that given two distinct points $i, j \leq r$ there are projections $p_{ij}\colon \FM_{n}(r) \to \FM_{n}(2)$.
Notice that  $\FM_{n}(2)$ is the sphere $S^{n-1}$, on which we have a standard volume form $\varphi \in \OmPA^{n-1}(\FM_{n}(2))$.
Given a graph $\Gamma \in \stG_{n}(k)$ with $l \ge 0$ internal vertices, let $E_{\Gamma}$ be its set of edges.
Setting $p_{ii}^*(\phi)=0$, we then define $\omega(\Gamma)$ to be the following integral along fibers (of the projection $\FM_{n}(k+l) \to \FM_{n}(k)$ which forgets all the points corresponding to internal vertices):
\begin{equation}
  \label{eq:6}
  \omega(\Gamma) \coloneqq \int_{\FM_{n}(k+l) \to \FM_{n}(k)} \bigwedge_{(i,j) \in E_{\Gamma}} p_{ij}^{*}(\varphi).
\end{equation}

\begin{rem}\label{rem:pa-forms}
  These integrals are the reason that we are forced to work with PA forms.
  Indeed, the projections $\FM_{n}(k+l) \to \FM_{n}(k)$ are not submersions \cite{LambrechtsVolic2014} in general, so we may not work with usual de Rham forms.
  However, they are semi-algebraic bundles.
\end{rem}

\begin{thm}[{\cite{Kontsevich1999,LambrechtsVolic2014}}]
  The morphisms defined above are quasi-iso\-morphisms of homotopy Hopf cooperads:
  \[ H(\FM_{n}) \xleftarrow{\sim} \stG_{n} \xrightarrow[\omega]{\sim} \OmPA(\FM_{n}). \]
\end{thm}

\subsection{The Kontsevich graph complex $\GC_n$}
\label{sec:gcn}

Let us also recall the definition of the graph complex $\GC_{n}$.
As a vector space, $\GC_{n}$ is spanned by infinite series of connected 1-vertex irreducible graphs (i.e. graphs that remain connected after deleting one vertex) consisting of internal vertices only, each having valence $\geq 2$.
Edges are directed, but the elements of $\GC_{n}$ must be invariant under edge reversal, with a coefficient $(-1)^{n}$ when an edge is reversed.
Thus we draw undirected edges in pictures, which are to be understood as the sum of an edge with its symmetric (with a sign).
\[
  \begin{tikzpicture}
    \node[int] (v1) at (0:1) {};
    \node[int] (v2) at (90:1) {};
    \node[int] (v3) at (180:1) {};
    \node[int] (v4) at (270:1) {};
    \draw (v1) edge (v2) edge (v3) edge (v4) (v2) edge (v3) edge (v4) (v3) edge (v4);
  \end{tikzpicture}
\]
Given a graph $\gamma \in \GC_{n}$ with $e$ edges and $v$ vertices, its cohomological degree is $vn -n-e(n-1)$, and minus that degree if one prefers homological conventions.
The differential is dual of the differential in $\Graphs_{n}$ and splits vertices in two vertices connected by an edge, summing over all possible ways of reconnecting edges incident to the initial vertex to the two vertices.

There is a (pre-)Lie algebra structure on $\GC_{n}$ given by insertion of graphs, denoted $\Gamma \star \Gamma'$.
This Lie algebra $\GC_{n}$ acts on each
$\Graphs_{n}(k)$ by Hopf cooperadic biderivations, again using insertion of graphs.
A more conceptual way of obtaining this structure is to realize that $\GC_{n}$ can be equivalently given as the deformation Lie algebra of a map from the $\Lie$ operad into an operad of graphs $\Gra_n^*$, which is essentially the dual of $\Graphs_n$, but with no differential or internal vertices.
  { The differential on $\Graphs_n$ arises from a twist by the Maurer--Cartan element}
\begin{tikzpicture}[every pin edge/.style={densely dotted}, every pin/.style={font=\footnotesize}, scale=1.5, baseline=-.1cm]
  \node[int] (i1) at (0,0) {};
  \node[int] (i2) at (0.3,0) {};
  \draw
  (i1) edge (i2)
  ;
\end{tikzpicture}.
Concretely, given $\mu\in \GC_n$ and $\Gamma $ a graph in the dual operad $\Gamma \in \Graphs_n^*$, the action is given by summing over (with appropriate signs) plugging $\mu$ into the vertices of $\Gamma$ and subtracting the insertion of $\Gamma$ into the vertices of $\mu$.
That is, if we denote $\mu_1$ the graph $\mu$ but with one of its vertices made external, then we have
\begin{equation}\label{eq:action of GCn on Graphs_n}
  \mu(\Gamma)=\mu_1\circ\Gamma-(-1)^{|\Gamma||\mu|}\Gamma\circ\mu_1-(-1)^{|\Gamma||\mu|}\Gamma\star\mu.
\end{equation}

We refer to~\cite[Appendix~I.3]{Willwacher2014} or~\cite[Equation~(7)]{Idrissi2016} for an explicit description of this action and~\cite[Proposition~3.2]{DolgushevWillwacher2015} for the deformation complex point of view.
Furthermore, $\GC_n$ can essentially be identified with the homotopy biderivations, and can be used to compute the homotopy automorphism space of the rationalizations of the little discs operads, see \cite{FresseWillwacher2020}.

\begin{rem}\label{rmk:val-gcn}
  It is possible to remove from $\Graphs_n$ and from $\GC_n$ the graphs with internal vertices of valence $2$. In fact, elsewhere in the literature, including the last author's papers, the notation $\Graphs_n$ and $\GC_n$ refers to the version without vertices of valence $2$, while our version is denoted $\Graphs_n^2$ or $\GC_n^2$ respectively.
  Since we only need the version above for this paper we omitted the superscript for the sake of cleaner notation.
\end{rem}

\subsection{Graphical models for $\FM_M$}
\label{sec:graph-models-fm_m}

The methods described in Section~\ref{sec:graphical-models} to build real models for $\Conf_{k}(\R^{n})$ were enhanced by some of the authors to describe real models for $\Conf_{k}(M)$ when $M$ is a closed orientable manifold~\cite{CamposWillwacher2016,Idrissi2016}.
We give here a quick account of the model found in the first reference.

The goal is to build a sequence of dgcas $\stG_{M}(k)$, equipped with an operadic right $\stG_{n}$-comodule structure when $M$ is framed.
Just like $\stG_{n}$, the space $\stG_{M}(k)$ is spanned by graphs with two types of vertices: external vertices, numbered $1, \dots, k$, and indistinguishable internal vertices of degree $-n$.
The edges are again undirected and of degree $n-1$.
Each vertex is decorated by zero, one, or more elements of the reduced cohomology $\tilde{H}(M)$, in other words, by an element of the free unital symmetric algebra $S(\tilde{H}(M))$, and each decoration increases the total degree of the graph.
Implicitly, we identify this algebra with a quotient of $S(H(M))$, where we identify the unit of the free algebra with the unit of $H(M)$.

Just like in $\Graphs_{n}$, tadpoles and double edges are allowed in $\stG_{M}$.
Also as before, we identify graphs with zero if they contain a univalent internal vertex, by which we mean an internal vertex with exactly one incident edge and no $\tilde{H}(M)$-decoration.
Furthermore, if the graph contains a connected component containing internal vertices only, that component is removed and replaced by a numeric prefactor, see below. So effectively, our graphs should be considered as not containing such connected components of internal vertices.

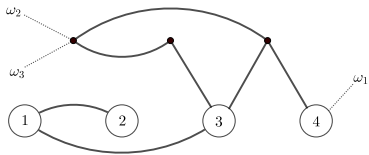
\begin{figure}[!h]
  \centering
  \begin{tikzpicture}[every pin edge/.style={densely dotted}, every pin/.style={font=\footnotesize}, scale=1.5]
    \node[ext] (e1) at (0,0) {1};
    \node[ext] (e2) at (1,0) {2};
    \node[ext] (e3) at (2,0) {3};
    \node[ext, pin={above right}:{$\omega_1$}] (e4) at (3,0) {4};
    \node[int, pin={above left}:{$\omega_2$}, pin={below left}:{$\omega_3$}] (i1) at (.5,1) {};
    \node[int] (i2) at (1.5,1) {};
    \node[int] (i3) at (2.5,1) {};
    \draw
    (e1) edge[bend right] (e2) edge[bend left] (e3)
    (e3) edge (i2) edge (i3)
    (i1) edge[bend right] (i2) edge[bend left] (i3)
    (i3) edge (e4)
    ;
  \end{tikzpicture}
  \caption{Illustration of an element in $\stG_{M}(4)$.}\vspace{-5pt}
\end{figure}

The differential $\delta$ is a sum of two parts, $\delta_{\mathrm{contr}}-\delta_{\mathrm{cut}}$:
\begin{itemize}
  \item The contracting part $\delta_{\mathrm{contr}}$ is the sum of all possible ways of contracting edges connected to an internal vertex, multiplying the decorations (in the free symmetric algebra).
        That summand also contracts tadpoles, which has the effect replacing a tadpole by a decoration of the incident vertex by the Euler class.

  \item The cutting part $\delta_{\mathrm{cut}}$ is the sum over all possible ways of cutting an edge and multiplying the endpoints of the edge by the diagonal class $\Delta_{M} \in H(M)^{\otimes 2}$ (we interpret a decoration by the unit $1\in H(M)$ as no decoration).
        Recall that given  a graded basis $\{ e_{i} \}$ of $H(M)$ the diagonal class is expressed as follows: if  $\{ e_{i}^{\vee} \}$ is the dual basis with respect to the Poincaré duality pairing (i.e.\ $\int_{M} e_{i} e_{j}^{\vee} = \delta_{ij}$) then
        \begin{equation}
          \label{eq:4}
          \Delta_{M} \coloneqq \sum_{i} (-1)^{|e_{i}|} e_{i} \otimes e_{i}^{\vee}.
        \end{equation}
\end{itemize}

Similar to $\Graphs_n$ and Equation \ref{eq:action of GCn on Graphs_n}, another way to obtain this differential is to consider the action of $\GC_n$ on $\Graphs_M^*$ by:

\[\mu(\Gamma) = -(-1)^{|\Gamma||\mu|}(\Gamma\circ\mu_1+\Gamma\star\mu).\]
Notice that the term $\mu_1\circ\Gamma$ in  Equation \ref{eq:action of GCn on Graphs_n} is not present here, since $\GC_n$ only acts on the right on $\Graphs_M^*$.

Following \cite[Definition 20]{CamposWillwacher2016}, we can consider the Lie algebra $\GC_M$, which is to $\Graphs_M$ as $\GC_n$ is to $\Graphs_n$ (note that since $\Graphs_M$ allows tadpoles, so must $\GC_M$).
There is a ``partition function'' $Z_M\in\GC_M$, which we interpret as a map from the pre-dual space $Z_{M}\colon {}^*\GC_M \to \mathbb R$ which assigns a real number to graphs with only internal vertices.
For example, if $\gamma$ is a graph with exactly one vertex and decorations $\alpha_{1}, \dots, \alpha_{k} \in \tilde{H}(M)$, then $Z_{M}(\gamma) = \int_{M} \alpha_{1} \wedge \dots \wedge \alpha_{k}$.
Then in the definition of $\stG_{M}$, a graph $\Gamma$ with a connected component $\gamma$ with only internal vertices is identified with $Z_{M}(\gamma) \cdot (\Gamma \setminus \gamma)$.

\begin{ex}\label{ex:partition}
  This partition function is, in general, difficult to compute, and few concrete examples are known.
  Generally, the tree part of the partition function encodes the real homotopy type of $M$, in the form of a (cyclic) homotopy commutative algebra structure on $H(M)$, see~\cite[Section~8]{CamposWillwacher2016}.
  This tree part can hence be computed relatively easily in many examples.
  Let us call the partition function ``trivial'' if it vanishes on all graphs of positive loop order.
  By Koszul duality, the differential $\delta$ on the free algebra generated by $\GC_M^*$ induces a dg-Lie algebra structure on $\GC_M$, and $Z_M$ can be seen as a Maurer--Cartan element in that dg-Lie algebra, see~\cite[Section~7.1]{CamposWillwacher2016}.
  It is known that $Z_{M}$ is gauge equivalent to a trivial one if $H^1(M)=0$ and of dimension $\ge 4$ by a simple degree counting argument, see~\cite[Lemma~54]{CamposWillwacher2016} or~\cite[Proposition~45]{Idrissi2016}, or if $M$ is a $2$-sphere by~\cite[Appendix~B]{CamposWillwacher2016}.
  Using different techniques, a result of similar nature was obtained for oriented surfaces after the submission of the first version of the present paper in~\cite{CamposIdrissiWillwacher2019}.

\end{ex}

Moreover, if $M$ is framed, or more generally if the Euler class of $M$ vanishes, then there is an operadic right $\stG_{n}$-comodule structure on $\stG_{M}$, given by subgraph collapsing (multiplying all the decorations of the collapsed subgraph in the process).

To define the quasi-isomorphism $\stG_{M} \to \OmPA(\FM_{M})$, one first chooses representatives of the cohomology of $M$ via an injective quasi-isomorphism of chain complexes $\iota : H(M) \to \OmPA(M)$.
We will generally suppress it from the notation, viewing $H(M)$ as a subcomplex of $\OmPA(M)$.
Then there exists a ``propagator'' \cite[Proposition~8]{CamposWillwacher2016}, a form $\varphi \in \Omega^{n-1}_{min}(\FM_{M}(2))$, which satisfies the following properties:
\begin{itemize}
  \item it is (anti-)symmetric, i.e.\ $\varphi^{21} = (-1)^{n} \varphi$;
  \item its differential $d\varphi$ is minus the pullback of $\Delta_{M}$ under the canonical projection $\FM_{M}(2) \to M^{2}$;
  \item its restriction to $\partial \FM_{M}(2)$, which is a sphere bundle over $M$, is a global angular form (i.e.\ its integral on every fiber is $1$);
        furthermore, when $n=2$, the restriction to each fiber of the circle bundle $\partial \FM_M(2) \to M$ must be a round volume form so that Kontsevich's argument below still works (see \cite[Proposition~8]{CamposWillwacher2016});

  \item for all $\alpha \in H(M)$, one has $\int_{y} \varphi(x,y) \alpha(y) = 0$.
\end{itemize}

This version of $\Graphs_M$, unlike the one in \cite{CamposWillwacher2016}, has tadpoles that we must handle specifically.
We choose a fixed form $\eta \in \Omega^{n-1}(M)$ satisfying $d\eta = -\sum_i (-1)^{\deg e_i} \iota(e_i)\wedge \iota(e_i^\vee) + \iota(E)$, where $\iota(E)$ is the chosen representative of the Euler class and the sum runs over a graded basis of $H(M)$, where $e_i^\vee$ denotes the dual basis to $e_i$ with respect to Poincaré duality.
Notice that if $n$ is odd $\iota(E)$ and $\eta$ can be chosen to be zero.

Then given a graph $\Gamma \in \stG_{M}(k)$ with $l$ internal vertices, $E_{\Gamma}$  its set of edges and $V_{\Gamma}$ its set of vertices; for a vertex $i \in V_{\Gamma}$, let $\alpha_{1}^{i}, \dots, \alpha_{r_{i}}^{i}$ be its decorations.
We define its image in $\OmPA(\FM_{M}(k))$ to be the following integral along fibers:
\begin{equation}
  \label{eq:7}
  \omega(\Gamma) \coloneqq \int_{\FM_{M}(k+l) \to \FM_{M}(k)} \bigwedge_{(i,j) \in E_{\Gamma}, i \neq j} p_{ij}^{*}(\varphi) \wedge \bigwedge_{(i, i) \in E_{\Gamma}} p_{i}^{*}(\eta) \wedge \bigwedge_{i \in V_{\Gamma}} (\iota(\alpha_{1}^{i}) \wedge \dots \wedge \iota(\alpha_{r_{i}}^{i})).
\end{equation}
Note that in general the pushforward is not defined for all PA forms but only for a subclass of \emph{trivial forms} and therefore the forms appearing in the integral must be chosen to be trivial \cite[Appendix C]{CamposWillwacher2016}.

\begin{thm}[{\cite[Theorems 42 and 25]{CamposWillwacher2016}}]
  The morphism described above defines a quasi-isomorphism of dgcas:
  \begin{equation*}
    \omega : \stG_{M}(k) \xrightarrow{\sim} \OmPA(\FM_{M}(k)).
  \end{equation*}
  If moreover $M$ is framed then (for a suitable choice of propagator, see \cite[Remark 9]{CamposWillwacher2016}) this is compatible with the operadic comodule structure, respectively over $\stG_{n}$ and $\OmPA(\FM_{n})$.
\end{thm}

In particular, note that
\begin{equation}
  \label{eq:8}
  A_{\mathrm{full}} \coloneqq \stG_{M}(1) \xrightarrow{\sim} \OmPA(M)
\end{equation}
is a real model for $M$
(if $M$ is not simply connected then this is a ``naive'' model, and we potentially need more information to recover the full real homotopy type of $M$).
For our purposes, we actually have to consider the ``tree part'' of $A_{\mathrm{full}}$, i.e.\ the sub-dgca $A$ of graphs with no loops, which is also a model for $M$~\cite[Lemma 52]{CamposWillwacher2016}:
\begin{equation}
  \label{eq:3}
  A \coloneqq \stG_{M}^{\mathrm{tree}}(1) \xrightarrow{\sim} \OmPA(M).
\end{equation}

Moreover, we have maps:
\begin{equation}
  \label{eq:9}
  A^{\otimes k} \to \stG_{M}(k),
\end{equation}
obtained by gluing the graphs at each external vertex, which represent the projections of Equation~\eqref{equ:projections}.

If $\dim M$ is even, then we have a canonical representative $E \in A$ of the Euler class of $M$.
Recall the graded basis $\{e_{i}\}$ and dual basis $\{e_{i}^{\vee}\}$ of $H(M)$.
Then our representative of the Euler class is given by a sum of graphs with two decorations:
\begin{equation}
  \label{eq:euler-class}
  E \coloneqq \sum_{i}
  (-1)^{\deg e_{i}}\
  \begin{tikzpicture}[baseline=(i.base), every pin edge/.style={densely dotted}, every pin/.style={font=\tiny}, pin distance=2mm]
    \node[ext, pin={above left}:{$e_i$}, pin={above right}:{$e_i^\vee$}] (i) {1};
  \end{tikzpicture}
\end{equation}
If $\dim M$ is odd then we merely set $E \coloneqq 0$ for notational consistency later.

\begin{rem}
  When $M$ is simply connected, it would be possible to replace $S(\tilde{H}(M))$ by a Poincaré duality model $A$ of $M$, as is done in~\cite{Idrissi2016}.
  However, this would add some technical complications due to the fact that there is no direct map $A \to \OmPA(M)$ in general, so we will not go down this path.
  Most of the constructions below would work similarly.
\end{rem}

\begin{rem}
In the above we adapt the sign conventions of \cite{LambrechtsVolic2014}, deviating slightly from \cite{CamposWillwacher2016}.
In particular, we ask our propagator $\phi$ to restrict to a unit volume form on the sphere bundle, for the standard orientation of the sphere bundle. This is opposite to the orientation it receives from being a boundary of $\FM_M(2)$.
This changes the sign in the differential of $\phi$. Note that the sign from the following computation with Stokes' formula, where the integral is over the fiber of $\FM_M(2)\to\FM_M(1)= M$,
\[
-1= \int_{\partial}\phi(z,w) = \int_{w}d\phi(z,w)
= -\sum_i (-1)^{|e_i|} \int_{w} e_i(z) \otimes e_i^\vee(w)
=-1.
\]
\end{rem}

\subsection{Equivariant graphical models}
\label{sec:equiv-graph-models}
Throughout the paper we will abbreviate $G = \SO(n)$ to shorten notation.
There is an action of $G$ on $\FM_{n}$ induced by the canonical action of $G$ on $\mathbb R^n$.
The framed Fulton--MacPherson operad $\FFM_{n}$ is then obtained as the ``framing product'' or ``semidirect product''~\cite{SalvatoreWahl2003} of $\FM_{n}$ with $G$:
\begin{equation}
  \FFM_{n} \coloneqq \FM_{n} \circ G = \{ \FM_{n}(k) \times G^{k} \}_{k \ge 0}.
\end{equation}

The action of the group $G = \SO(n)$ on $\FM_{n}$ is not directly apparent on the model $\Graphs_{n}$.
Let us now describe it.
Consider the abelian Lie algebra
\begin{equation}\label{eq:lie-alg-so}
  \mathfrak{g} = \bigoplus_{i \ge 0} \pi_{i}(\SO(n)) \otimes_{\Z} \R
  =\begin{cases}
    \R P_{3} \oplus \R P_{7} \oplus \dots \oplus \R P_{2n-7} \oplus \R \chi_{n-1}, & \text{if $n$ is even}; \\
    \R P_{3} \oplus \R P_{7} \oplus \dots \R P_{2n-7} \oplus \R P_{2n-3},          & \text{if $n$ is odd};
  \end{cases}
\end{equation}
where the elements $P_{i}$, living in degree $i$, correspond to the Pontryagin classes and $\chi_{n-1}$, living in degree $n-1$, corresponds to the Euler class.
The action of $G$ on $\FM_n$ may be described by an $L_\infty$-action of $\alg g$ on $\Graphs_n$.
This $L_\infty$ action has been identified in \cite[Theorem 7.1]{KhoroshkinWillwacher2017}, and factors through the action of the graph complex $\GC_{n}$ from Section~\ref{sec:graphical-models}.
Concretely, the graph complex $\GC_n$ acts on $\Graphs_n$ by cooperadic bi-derivations, i.e., compatibly with the Hopf cooperad structure.
Then an $L_\infty$-action of $\alg g$ on $\Graphs_n$ that factors through $\GC_n$ may be described by a Maurer--Cartan element:
\begin{equation}
  \label{eq:mc-elt-m}
  m\in C_{CE}(\alg g) \hotimes \GC_n =H(\BG) \hotimes \GC_n.
\end{equation}
In \cite{KhoroshkinWillwacher2017} an explicit formula in terms of integrals over configuration spaces is given for $m$.
Furthermore, the gauge equivalence type of $m$ is identified.
\begin{thm}[Theorems 1.2, 1.3 of \cite{KhoroshkinWillwacher2017}] \label{thm:MCform}
  The Maurer--Cartan element $m\in H(\BG)\hotimes\GC_n$ is gauge equivalent to
  \begin{equation}\label{equ:m formula}
    \begin{cases}
      -E \tadp, & \text{for $n\geq 2$ even}; \\
      \sum_{j\geq 1}
      \frac {p_{2n-2}^{j}}{2^{2j+1}}
      \frac{1}{(2j+1)!}
      \begin{tikzpicture}[baseline=-.65ex]
        \node[int] (v) at (0,.5) {};
        \node[int] (w) at (0,-0.5) {};
        \draw (v) edge[bend left=50] (w) edge[bend right=50] (w) edge[bend left=30] (w) edge[bend right=30] (w);
        \node at (2,0) {($2j+1$ edges)};
        \node at (0,0) {$\scriptstyle\cdots$};
      \end{tikzpicture},
               & \text{for $n\geq 3$ odd};
    \end{cases}
  \end{equation}
  where (up to a normalization factor) $E\in H(\BG)$ is the Euler class and $p_{2n-2}\in H(\BG)$ is the top Pontryagin class.
\end{thm}

\begin{rem}\label{rem:prefactor}
  Note that this theorem the graphs that we write are coinvariants but $\GC_n$ is defined in terms of invariants.
  The factor $\frac{1}{2^{2j+1}(2j+1)!}$
  comes from the symmetries of the graph (permuting edges and changing edge orientations), and appears essentially due to an implicit identification between invariants and coinvariants, see~\cite[Remark~7.4]{KhoroshkinWillwacher2017}.
  As explained in~\cite[Lemma~7.3]{KhoroshkinWillwacher2017}, the actual value of the Maurer--Cartan element on the graph with two vertices and $2j+1$ edges (for odd $n$) is $p^{j}_{2n-2}$.
  This distinction is important when one wants to write down the differential of $\BGraphs_{n}^{m}$ defined below.
\end{rem}

We may lift the $L_\infty$-action of $\alg g$ on $\Graphs_n$ to an honest dg Lie action of a resolution
\begin{equation}
  \label{eq:hat-g}
  \hat{\alg g} \xrightarrow{\sim} \alg g.
\end{equation}
Concretely, $\hat{\alg g}$ can be taken to be a quasi-free Lie algebra generated by the reduced homology $\tilde{H}_*(\BG)$ (notice that we use the ${\scriptstyle *}$ subscript to refer to homology but no superscript for cohomology), i.e.\ the cobar construction on the Koszul dual (cocommutative coalgebra) of $\mathfrak{g}$.
We furthermore have the identification $H(G) = \U \mathfrak{g}^*$, where $\U$ denotes the universal coenveloping coalgebra, which is a commutative and cocommutative Hopf algebra.
We similarly define the (commutative but not cocommutative) Hopf algebra:
$$\hat H(G) \coloneqq \U \hat{\alg g}^*.$$
Note that this is isomorphic as a Hopf algebra to the associative bar construction on the cohomology $\tilde H(\BG)$, equipped with the shuffle product.
It is possible to take a smaller resolution given the particular expression of $m$ in Equation~\eqref{equ:m formula} and in fact for even $n$ one can take $\hat{\mathfrak g}=\mathfrak g$, thus obtaining a cocommutative model for $\hat H(G)$ cf. \cite[Section 9.1]{KhoroshkinWillwacher2017}.

Via the action of $\hat{\alg g}$ we may equip $\Graphs_n$ with an $\hat H(G)$-coaction.
This is the model of \cite{KhoroshkinWillwacher2017} of $\FM_n$ as an operad in $G$-spaces.
To obtain a model for the framed little discs operads $\FFM_n$ one can take the semidirect product
\[
  \Graphs_n \circ \hat H(G).
\]

Unfortunately, there is no direct map between the above models and the forms $\OmPA(\FM_n^{\fr})$.
The construction in \cite{KhoroshkinWillwacher2017} instead yields a zigzag of quasi-isomorphisms.
In this paper we shall also need intermediate objects in this zigzag.
In particular, one model for the equivariant forms on a $G$-space $X$ is given by the following dgca~\cite[Section~4]{KhoroshkinWillwacher2017}:
\begin{equation}
  \Omega'^{s}_{G}(X) \coloneqq \Tot \OmPA(G^{\bullet} \times X) = \int_{k \in \Delta_{+}} \OmPA(G^k \times X) \otimes \Omega_{\mathrm{PL}}\Delta^{k},
\end{equation}
where $G^{\bullet} \times X = B_{\bullet}(*, G, X)$ is the simplicial bar construction and $\Tot$ is actually the fat totalization, i.e.\ the end is only over $\Delta_{+}$, the cosimplicial category with objects $[k] = \{0,1, \dots, k\}$ ($k \ge 0$) and morphisms are \emph{strictly} increasing maps.
In particular,
\begin{equation}
  B_{G} \coloneqq \Omega'^{s}_{G}(*) = \Tot \OmPA(G^{\bullet})
\end{equation}
is a model for $* \sslash G = \BG$, and there is a quasi-isomorphism of dgcas $H(\BG) \to B_{G}$.
The dgca ${\Omega'}_G^{s}(X)$ is a $B_{G}$ module.
Notice that the usual monoidal structure on $B_G$-modules is not homotopy invariant.
To correct this, one can instead consider its free resolution given by the two-sided bar construction:
\begin{equation}\label{eq:Omega^s_G}
  \Omega^{s}_{G}(X) \coloneqq B_{G} \otimes^h_{B_{G}} \Omega'^{s}_{G}(X) \coloneqq \biggl( \bigoplus_{k \ge 0} B_{G} \otimes (B_{G}[1])^{\otimes k} \otimes \Omega'^{s}_{G}(X), d \biggr).
\end{equation}
Note that while the two-sided bar construction is associative, it is not symmetric but only symmetric up to homotopy.

Such associativity implies that the functor $\Omega^{s}_{G}$ is oplax monoidal, for the same reason the ordinary differential forms functor is oplax monoidal; the natural transformation is obtained by multiplication of pullbacks of the projections $X\times Y\to X$ and $X\times Y\to Y$. However, note that $\Omega^{s}_{G}$ is not oplax \emph{symmetric} monoidal.

The first step in~\cite{KhoroshkinWillwacher2017} is to find a model for the equivariant forms on $\FM_{n}$.
The model is denoted $\pdu\BGraphs^{m}_{n}$.
As a graded vector space,
\begin{equation}\label{def:bgraphs}
  \pdu\BGraphs_{n}^{m} = (\stG_{n} \otimes H(\BG), \, d + d_m).
\end{equation}
The commutative algebra structure is defined term-wise.
The cooperad structure of $\Graphs_n$ induces on $\BGraphs_n^m$ a (strict) cooperad structure in dgcas under $H(\BG)$, where the monoidal product is $\otimes_{H(\BG)}$.
The differential is a sum of two terms.
The first is the differential from $\stG_{n}$, i.e.\ it contracts edges incident to internal vertices.
The second is the twist by the Maurer--Cartan element $m \in H(\BG) \hat\otimes \GC_{n}$, which is gauge equivalent to the element described in Theorem~\ref{thm:MCform}.
Here, the action of $H(\BG) \hat\otimes \GC_{n}$ is given by multiplication of the $H(\BG)$ pieces and the  $\GC_{n}$ action on $\Graphs_n$ as in Section \ref{sec:gcn}.
In particular, for even $n$, the twist contracts a tadpole and multiplies the decoration by the Euler class, while for odd $n$, the twist contracts multiple adjacent edges and multiplies the decoration by a power of the Pontryagin class.

Then we have a direct quasi-isomorphism (compare with~\cite[Theorem~6.7]{KhoroshkinWillwacher2017} which uses a different model):
\begin{equation}
  \label{eq:omega-equivar}
  \omega_{\mathrm{equivar}} : \pdu\BGraphs_{n}^{m} \xrightarrow{\sim} \Omega^{s}_{G}(\FM_{n}) = B_G \hotimes_{B_G} \Omega'^{s}_G(\FM_n).
\end{equation}
This map is given by an integration procedure similar to the ones of Section~\ref{sec:graphical-models}.
These integrals use an explicit ``equivariant propagator'':
\[\Omega_{sm}^C \in (S(\alg{so}_{n}^*[-2])\otimes \OmPA(S^{n-1}))^G,\]
see Appendix~\ref{sec:explicit_prop} (note that the formulas in \cite[Appendix~A]{KhoroshkinWillwacher2017} are in the toric model rather than the Cartan model that we use).
\todo{Check that item}
The edges in a graph are sent to the image of $\Omega_{sm}^C$ in $\Omega^{\prime s}_{G}(\FM_n(2))$ as we now explain.
This image is defined using a choice of formality dgca quasi-isomorphism $H(\BG) = S(\mathfrak{so}_n^*[-2]) \to B_G$.
Then, we take the wedge product with forms in $\OmPA(S^{n-1})=\OmPA(\FM_n(2))$.
Finally, to fully define $\omega_{\mathrm{equivar}}$ on all graphs, we integrate away internal vertices.

The second step in~\cite{KhoroshkinWillwacher2017} is to recover a model for $\OmPA(\FM_{n})$ together with its action of $G$.
There is a (homotopy) pullback square:
\begin{equation}
  \begin{tikzcd}
    \FM_{n} \ar[r] \ar[d] \ar[dr,phantom,"\lrcorner" near start] & EG \ar[d] \\
    \FM_{n} \sslash G \ar[r] & \BG.
  \end{tikzcd}
\end{equation}
Therefore, by the ``pullback-to-pushout principle''~\cite[Theorem~2.4]{Hess2007} and the fact that $\BG$ is simply connected, a model for $\FM_{n}$ is given by a pushout of the models of the three other spaces in the diagram.
The model of $\FM_{n} \sslash G$ is $\pdu\BGraphs_{n}^{m}$ defined above.
The model for $\BG$ is merely $H(\BG)$.
Finally, the model for $EG$ is given by the ``Koszul complex'',
\begin{equation}\label{eq:koszul-complex}
  K_G \coloneqq (H(\BG) \otimes H(G), d_{\kappa}),
\end{equation}
defined as follows.
For any compact Lie group $G$, we have $H(\BG) = \R[\alpha_{1}, \dots, \alpha_{r}]$ and $H(G) = \Lambda(\beta_{1}, \dots, \beta_{r})$ for some classes with $\deg \alpha_{i} = \deg \beta_{i} + 1$.
Then the complex $K_G$ is equipped with the differential such that $d_{\kappa}(\beta_{i}) = \alpha_{i}$.
There is a homotopy $h_{\kappa}$ such that $h_{\kappa}(\alpha_{i}) = \beta_{i}$, and one checks easily that
\begin{equation}
  \label{eq:k-acyclic}
  d_{\kappa}h_{\kappa} + h_{\kappa}d_{\kappa} = {\id_{K_G}} - \varepsilon(-) \cdot 1,
\end{equation}
where $\varepsilon : K_G \to \R$ is the augmentation.

The collection $\pdu\BGraphs_{n}^{m}$ is a (strict) cooperad in $H(\BG)$-modules that defines an equivariant model of $\FM_n$.
We may then apply the lax monoidal functor $K_G \otimes_{H(\BG)} (-)$ to obtain a homotopy Hopf cooperad (compare with the analogous construction in~\cite[Lemma~4.8]{KhoroshkinWillwacher2017}):
\begin{equation}\label{eq:bn}
  B_{n} \coloneqq K_G \otimes_{H(\BG)} \pdu\BGraphs_{n}^{m}.
\end{equation}
Using the pullback-to-pushout principle, this forms a dgca model of $\FM_n$.
We would thus like to connect it to $\OmPA(\FM_{n})$, taking the action of $H(G)$ into account.
However, there is no direct map that is compatible with the Hopf cooperad structure, and the zigzag is built using the following method.

Like all Lie groups, $G$ is formal as a space, and there exists a direct quasi-isomorphism of dgcas $H(G) \xrightarrow{\sim} \OmPA(G)$ (defined by choosing any closed representative of the classes $\beta_{i}$ above).
However, $\OmPA(G)$ is not a Hopf algebra, and the category of $\OmPA(G)$-modules is not a symmetric monoidal category, which would cause problems later.
One can strictify $\OmPA(G)$ into a non-unital complete bialgebra in complete vector spaces using the $W$-construction~\cite[Section~3]{KhoroshkinWillwacher2017}.
(It can be made unital using constructions from the same section but we will not use it.)
Let $I \coloneqq \Omega_{PL}(\Delta^{1}) = \R[t,dt]$ be a path object for $\R$ in the model category of dgcas.

\begin{defi}
  Let $\mathcal{L}$ be the category whose objects are rooted linear trees with some distinguished edges (called ``cut edges'').
  Morphisms are generated by isomorphisms of trees, contractions of non-distinguished edges, and marking some internal edges as distinguished.

  Given an element $T \in \mathcal{L}$, we let $\pi_0(T)$ be the set of connected components of $T$ with the cut edges removed, $T_{\text{root}} \in \pi_0(T)$ the connected component containing the root, $V_S$ the set of vertices of some $S \in \pi_0(T)$, and $E^{nc}_T$ be the set of non-cut edges of $T$.
\end{defi}

\begin{rem}
  In~\cite[Section 3.7]{KhoroshkinWillwacher2017}, the analogous categories $\mathcal{T}_{S}$ (for sets of leaves $S$) of non-linear trees are used because the construction is used more generally for operads.
\end{rem}

\begin{defi}\label{def:a-g}
  The $W$-construction of $\OmPA(G)$ is given by the end:
  \begin{equation}
    \label{eq:a-g}
    A_G \coloneqq W \OmPA(G) = \int_{T \in \mathcal{L}} \bigotimes_{S \in \pi_0(T)} \OmPA(G^{V_{S}}) \otimes I^{\otimes E^{nc}_{T}}.
  \end{equation}
  Isomorphisms act in the obvious way, edge contractions use the product of $G$ and the evaluation $\operatorname{ev}_0 : \R[t,dt] \to \R$, and the edge cutting uses $\operatorname{ev}_{1} : \R[t,dt] \to \R$.
  The product of $A_G$ is simply induced from the products of $\OmPA(G^{V_T})$ and $I$.
  The coproduct of some $\varphi = \{ \varphi(T) \}_{T \in \mathcal{L}}$ is given on a pair of trees $(T,T')$ by gluing the two trees end to end with a cut edge and evaluating $\varphi$ on the result.
\end{defi}

The cochain complex $A_G$ is equipped with a complete descending filtration in which $\varphi \in F^p A_G$ if for any tree $T$ with less that $p$ vertices, $\varphi(T)=0$.
One then checks that $A_G$ is a complete Hopf algebra. Indeed, note that the cocomposition might contain infinite sums. However, the product $A_G\otimes A_G \to A_G$ extends to $A_G\hat \otimes A_G$, since for a given tree $T$, the value $(\varphi \cdot \psi) (T)$ depends only on the finitely many possible ways one can decompose $T$.

There is a canonical quasi-isomorphism of dgcas $A_G \qiso \OmPA(G)$ and moreover, by Proposition \ref{prop:hopf-formal}, $H(G) \xrightarrow{\sim} \OmPA(G)$ can be lifted to a quasi-isomorphism of Hopf algebras $H(G) \qiso A_G$.
Similarly, one can consider the $W$ resolution of a homotopy $\OmPA(G)$-comodule $\OmPA(X)$, turning it into a complete $A_G$-comodule $\modo_{A_G}(X)$:
\begin{multline}\label{eq:mod-a-g}
  \modo_{A_G}(X) = W \OmPA(X) \coloneqq \\
  \int_{T \in \mathcal{L}} \Bigl( \bigotimes_{\substack{S \in \pi_0(T) \\ \text{root} \not\in S}} \OmPA(G^{V_S}) \Bigr) \otimes \OmPA(G^{V_{T_\text{root}} \setminus \text{root}} \times X) \otimes I^{\otimes E^{nc}_{T}}.
\end{multline}

Let us now consider the simplicial resolution of $\FM_{n}$ as a $G$-space obtained by considering the bar complex:
\begin{equation}\label{eq:hat-fm}
  \widehat{\FM}_{n} \coloneqq G \times G^{\bullet} \times \FM_{n} \twoheadrightarrow \FM_{n}.
\end{equation}
The bar construction is oplax monoidal (using the diagonal of $G$), so that the operad structure of $\FM_n$ induces a structure of homotopy operad in $G$-modules on $\widehat{\FM}_n$.
We can then apply $\modo_{A_G}$ and totalization to obtain a homotopy cooperad in $A_G$-comodules $\Tot \modo_{A_G}(\widehat{\FM}_n)$.

Recall that $B_n = K_G \otimes_{H(\BG)} \pdu\BGraphs_{n}^{m}$ is the cooperad defined in Equation~\eqref{eq:bn} that forms a model for $\FM_n$ as a (strict) cooperad in $H(G)$-comodules.
More precisely, if we unpack the proof, we get a zigzag of quasi-isomorphisms compatible with the $G$-action, which involve terms that we are going to explain next:
\begin{equation}\label{eq:zigzag-bn-omega}
  B_n \xrightarrow{\sim} \Tot \modo_{A_G}(\widehat{\FM}_n) \xleftarrow{\sim} \modo_{A_G}(\FM_{n}) \xleftarrow{\sim} \OmPA(\FM_{n}).
\end{equation}

In this equation, we have the following notation and maps:
\begin{itemize}
  \item Applying the oplax monoidal functor $\modo_{A_G}$ to the operad $\FM_n$ arity-wise, we obtain a homotopy Hopf cooperad $\modo_{A_G}(\FM_n)$ in complete $A_G$-comodules.
        The last map $\modo_{A_G}(\FM_n) \leftarrow  \OmPA(\FM_n)$ is simply the resolution.
  \item Since $\widehat{\FM}_{n}$ projects onto $\FM_n$,
        the contravariance of $\OmPA$ gives the second map $\Tot \modo_{A_G}(\widehat{\FM}_{n}) \leftarrow \modo_{A_G}(\FM_n)$ which is a quasi-isomorphism.
  \item Finally, the first map is induced from $\omega_{\mathrm{equivar}}$~\eqref{eq:omega-equivar} in three steps:
        \begin{enumerate}
          \item Using $\omega_{\mathrm{equivar}}$, we map $B_n$ to $K_G \otimes_{H(\BG)} \Omega^s_G(\FM_n)$.
          \item Then we have a quasi-isomorphism:
                \[\Upsilon_G : K_G \to \modo_{A_G}(*) \to \Tot \modo_{A_G}(\widehat{*}),\]
                where the first map is induced by $\Upsilon_G$ (Proposition~\ref{prop:hopf-extension}) and $\widehat{*} = G^{\bullet + 1}$ is contractible.
          \item Finally, we map the pushout
                \[\Tot \modo_{A_G}(\widehat{*}) \otimes_{H(\BG)} \Omega^s_G(\FM_n) \simeq \Tot \modo_{A_G}(\widehat{*}) \otimes_{B_G} \Omega'^{s}_G(\FM_n)\]
                to $\Tot \modo_{A_G}(\widehat{\FM}_{n})$ using the morphism of \cite[Proposition~4.7]{KhoroshkinWillwacher2017}.
                On the first factor, we just use the dual of the projection $X \to *$.
                On the second factor, we project first to $\Omega^{\prime s}_G(\FM_n)$, then we map a collection $\{\varphi(k)\}_{k \in \Delta_+}$ (where $\varphi(k) \in \OmPA(G^k \times \FM_n) \otimes \Omega_{PL}(\Delta^k)$) to a collection $\{\psi(\xi)\}_\xi$ indexed by nested linear trees (linear trees whose edges are indexed by linear trees themselves).
                The value $\psi(\xi)$ on such a nested tree is simply obtained by applying $\varphi$ to the number given by forgetting the nesting.
        \end{enumerate}
\end{itemize}

The leftmost object in this diagram is quasi-isomorphic to $\Graphs_n$, with the $\hat H(G)$-coaction considered in the beginning of this section.
To see this, we may define the resolved Koszul complex $\hat{K}_{G} \coloneqq (\U\hat{\mathfrak{g}}^{*} \otimes H(\BG), d_{\hat{\kappa}})$ to get a homotopy cooperad in complete $\hat{H}(G)$-comodules (with the cooperad structure induced by that of $\BGraphs_n^m$):
\begin{equation}\label{eq:definition of Bn}
  \hat{B}_{n} \coloneqq \hat{K}_{G} \otimes_{H(\BG)} \pdu\BGraphs_{n}^{m}
\end{equation}
Then we have an explicit zigzag of quasi-isomorphisms of Hopf cooperads in $\hat H(G)$-comodules (compare with~\cite[Proposition~9.1]{KhoroshkinWillwacher2017}):
\begin{equation}
  \label{eq:zigzag-bn}
  B_n \qiso \hat B_n \qisor \stG_n.
\end{equation}
The first map is the inclusion.
The second map\najib{check this?} is the unique morphism whose projection onto comodule cogenerators is given by the inclusion $\stG_n \to H(\BG) \otimes \stG_n$.
Combining zigzags \eqref{eq:zigzag-bn-omega} and \eqref{eq:zigzag-bn} we obtain an $A_G$-equivariant equivalence of homotopy Hopf cooperads:
\begin{equation}\label{eq:zigzag-number-3}
  \Graphs_{n} \simeq \OmPA(\FM_{n}).
\end{equation}

The final step in~\cite{KhoroshkinWillwacher2017} is to show that $(K_G \otimes_{H(\BG)} \pdu\BGraphs_{n}^{m}) \circ H(G)$ can be connected to $\OmPA(\FFM_{n})$ by a zigzag of quasi-isomorphisms.
This uses explicit $W$-resolutions of (co)modules over (co)operads.
We use a similar construction in the proof of Theorem~\ref{thm:model-ffmm}.

\begin{rem}
  \label{rem:oriented-unoriented}
  In~\cite{KhoroshkinWillwacher2017}, the group considered is the full orthogonal group $\mathrm{O}(n)$.
  This adds difficulties, as $\mathrm{O}(n)$ is disconnected.
  Compared to what we have written here, there is an additional step required, consisting of considering invariants under the action of $\mathrm{O}(n)/\SO(n) \cong \{ \pm 1 \}$.
  In what follows, we will only consider the $\SO(n)$-action and oriented manifolds for simplicity.
  In order to obtain results for unoriented manifolds, one should consider the unoriented frame bundle $\Fr_{M}^{\mathrm{unor}}$ instead of the oriented frame bundle $\Fr_{M}$, consider $\SO(n)$-equivariant forms, and take the extra step of considering invariants under the action of $\{\pm 1\}$.
\end{rem}

\section{The fiber-wise little discs operad}
\label{sec:fibered-little-discs}

\subsection{Motivation}
\label{sec:motiv-defin}

Let $M$ be an oriented manifold of dimension $n$.
In general, if $M$ is not framed, then the spaces $\FM_{M}$ do not form a right $\FM_{n}$-module.
Indeed, in order to insert an infinitesimal configuration in a point $x \in M$, one needs to identify the tangent space $T_{x}M$ with $\R^{n}$.
If $M$ is not framed, there is no way to do this coherently for all $x \in M$.

To correct this, we build a new operad $\FM_{n}^{M}$ in topological spaces over $M$, which we call the fiber-wise Fulton--MacPherson operad over $M$.
The operad is defined such that the fiber over the map $\FM_{n}^{M}\to M$ at a point $x$ is (essentially) the Fulton--MacPherson-compactified configuration space of points in the tangent space $T_xM$.
Given such an element, one can insert the infinitesimal configuration into the tangent space at $x$ using the given frame, so that we have composition maps:
\begin{equation}
  \label{eq:1}
  \circ_{i} : \FM_{M}(r) \times^{i}_{M} \FM_{n}^{M}(s) \to \FM_{M}(r+s-1),
\end{equation}
where the pullback on the LHS is obtained by considering the projection $p_{i} : \FM_{M}(r) \to M$ which forgets all but the $i$-th point.

\subsection{Definition}

Let us now describe this operad more precisely and in more generality.
As before, let $G = \SO(n)$ and let $Y \to B$ be a principal $G$-bundle -- the example that we will care the most about being the oriented frame bundle $Y = \Fr_{M}$ over $B = M$.
We define an operad $\FM_{n}^{Y \to B}$ by:
\begin{equation}
  \label{eq:2}
  \FM_{n}^{Y \to B}(r) \coloneqq Y \times_{G} \FM_{n}(r),
\end{equation}
keeping in mind that in this notation, the index in $\times_{G}$ denotes a quotient by the action of $G$, not a pullback.
This is an operad in the category $\Top / B$ of spaces over $B$.
The unit is given by the identity $B \to \FM_{n}^{Y \to B}(1) = Y/G = B$, and the composition is defined by:
\begin{align*}
  \circ_{i} : \FM_{n}^{Y \to B}(r) \times_{B} \FM_{n}^{Y \to B}(s) & \to \FM_{n}^{Y \to B}(r+s-1) \\
  ([y,c], [y', c'])                                                & \mapsto [y, c \circ_{i} (y/y' \cdot c')],
\end{align*}
where $Y \times_{B} Y \to G$, $(y,y') \mapsto y/y'$ is defined using the principal bundle structure, and the action of $G$ on $\FM_{n}$ is by rotations.

Fix some Riemannian metric on $M$, which allows us to define the notion of ``orthonormal basis'' in tangent spaces of $M$.
In the special case that $Y=\Fr_{M}$ is the oriented orthonormal frame bundle over $M$, we abbreviate the operad defined above to:
\begin{equation}
  \FM_n^M \coloneqq \FM_n^{\Fr_M\to M}.
\end{equation}

The object $\FM_M$ carries a structure which we call \textbf{right operadic multimodule} for $\FM_n^M$.
Concretely, we have the projections of Equation~\eqref{equ:projections}, and natural ``insertion'' operations
\begin{equation}
  \FM_M \circ_M \FM_n^M \to \FM_M
\end{equation}
that are compatible with the operad structure on $\FM_n^M$.
Here, $\circ_M$ denotes a variation of the usual plethysm using that  $\FM_M(r)$ comes with $r$ maps to $M$.
Concretely,  $\FM_M \circ_M \FM_n^M\subset  \FM_M \circ \FM_n^M$ corresponds to the tuples $(\underbrace{c}_{\FM_M(k)};a_1,\dots,a_k)$ such that for all $i=1,\dots,k$, the projection $p_i(c)$ agrees with the corresponding location of $a_i$.
In particular, we note that $\FM_M$ is not an operadic right $\FM_n^M$ module in $\Top/M$.

Furthermore we note that the notion of right operadic multimodule has a natural ``homotopy'' equivalent, similarly to the notion of homotopy operads and modules recalled in Section \ref{sec:htpyoperads}.

\begin{rem}
  As explained in Remark~\ref{rem:oriented-unoriented}, if we were dealing with an unoriented manifold, we would need to look at the principal $\mathrm{O}(n)$-bundle given by the unoriented frame bundle $\Fr_{M}^{\mathrm{unor}}$ in the definition of $\FM_{n}^{M}$.
\end{rem}

\begin{rem}
  It is possible to give the following interpretation of these algebraic structures.
  On any topological space $X$, there exists a unique comonoid structure, with counit the unique map $\varepsilon : X \to *$, and coproduct $\Delta : X \to X \times X$ given by the diagonal $\Delta(x) = (x,x)$, which is automatically cocommutative.
  A right (or left) $X$-comodule $M$ is nothing but a space $M$ equipped with a map $f : M \to X$, as the coaction $M \to M \times X$ is forced to be of the form $m \mapsto (m,f(m))$ by the counit axiom.

  Such a comonoid $X$ naturally defines a cooperad $\underline{X}$ concentrated in arity $1$.
  An operadic left $\underline{X}$-comodule is the same thing as a $\Sigma$-collection $F = \{ F(k) \}_{k \geq 0}$ equipped with $\Sigma_{k}$-invariant maps $f^{k} : F(k) \to X$ for all $k \ge 0$.
  Similarly, an operadic right $\underline{X}$-comodule is the same thing as a $\Sigma$-collection $G = \{ G(k) \}_{k \geq 0}$ equipped with maps $g^{k}_{i} : G(k) \to X$ compatible with the $\Sigma_{k}$-action for all $k \ge 1$ and $1 \le i \le k$.
  Given a left $\underline{X}$-comodule $F$ and a right $\underline{X}$-comodule $G$, one can define their usual composition product over the cooperad $\underline{X}$ using pullbacks:
  \[ (G \circ^{\underline{X}} F)(k) \coloneqq
    \bigsqcup_{r\geq 0} \left(\bigsqcup_{l_{1} + \dots + l_{r} = k} \bigl( G(r) \times_{X^{r}} (F(l_{1}) \times \dots \times F(l_{r})) \bigr) \times_{\Sigma_{l_{1}} \times \dots \times \Sigma_{l_{r}}} \Sigma_k\right)_{\Sigma_r}. \]
  Note that a left $\underline{X}$-comodule structure induces a right $\underline{X}$-comodule structure by setting $g^{k}_{i} = f^{k}$ for all $i$.
  An operad in the category of spaces over $X$ (such as $\FM_{n}^{M}$) is then the same thing as a left $\underline{X}$-comodule $P$ equipped with a monoid structure in the category of $\underline{X}$-bicomodule.
  A right operadic multimodule $F$ (such as $\FM_{M}$) is a right $\underline{X}$-comodule equipped with a composition map compatible with the operadic structure maps of $P$:
  \[ F \circ^{\underline{X}} P \to F. \]
\end{rem}

\subsection{A model for the frame bundle}
\label{sec:model-frame-bundle}

The oriented frame bundle $\Fr_{M}$, like all principal $G$-bundles, fits in a pullback diagram
\begin{equation}
  \begin{tikzcd}
    \Fr_M \ar[r] \ar[d] \ar[dr, phantom, "\lrcorner" near start] & EG\arrow{d} \\
    M \arrow{r} & \BG
  \end{tikzcd},
\end{equation}
where $M \to \BG$ classifies the (oriented) tangent bundle of $M$.

Recall that we take $A = \stG^{\mathrm{tree}}_{M}(1)$ as a model for $M$ (see Equation~\eqref{eq:3}).
The algebraic version of the above square is the following pushout diagram:
\begin{equation}
  \begin{tikzcd}
    \Fr_M^{\mathrm{alg}} \ar[dr, phantom, "\lrcorner" near start] & K_G \ar[l] \\
    A \ar[u] & H(\pdu\BG)\arrow{l}\arrow{u}
  \end{tikzcd}
\end{equation}
where $K_G = (H(\BG) \otimes H(G), d_{\kappa})$ is the almost acyclic ``Koszul complex'', see Section~\ref{sec:equiv-graph-models}.

The map $H(\BG) \to A$ sends the Pontryagin classes (resp., the Euler class if $n$ is even) to graphs with a single external vertex, decorated by the representative of the respective Pontryagin class (resp., Euler class) of $M$ given by the initial choice of map of chain complexes $H(M) \to \OmPA(M)$.
This is done so that real dgca model for $\FM_M$ as right $\FM_n^M$-module is compatible with the differentials.

It follows that the algebraic model for the frame bundle is
\begin{equation}
  \Fr_M^{\mathrm{alg}} \coloneqq A \otimes_{H(\BG)} K_G = (A \otimes H(G), d),
\end{equation}
where the differential takes a generator from $H(G)$ and attaches to the external vertex the corresponding Pontryagin/Euler class.

However, we do not have a map $\Fr_M^{\mathrm{alg}} \to \OmPA(\Fr_M)$ directly compatible with the $G$-action.
For this reason we will also consider a ``resolution'' of the frame bundle, namely its $W$ construction as right $G$-space $\Fr^W_M$.
This $W$-construction comes with a natural $WG$-action.

There is a quasi-isomorphism of complete Hopf algebras $\upsilon_G : H(G) \to A_G$ (where $A_G \coloneqq W \OmPA(G)$ was defined in Equation~\eqref{eq:a-g}), see Proposition~\ref{prop:hopf-formal}.
Let us take the minimal model for $\hat{\alg g}$ as in Equation~\eqref{eq:hat-g} (see \cite[Section 9.1]{KhoroshkinWillwacher2017}), which, as a coenveloping coalgebra of a Lie coalgebra, is free (but not cofree).
The map $H(G) \to A_G$ can be extended to a quasi-isomorphism of complete Hopf algebras:
\begin{equation}\label{eq:hat-h-g}
  \hat{H}(G) = \U \hat{\alg g}^* \to A_G.
\end{equation}
Indeed, $H(G) \to \hat{H}(G)$ is an acyclic cofibration, since it is obtained by adding generators successively; and $A_G$ is fibrant, since it is quasi-cofree on primitive elements, with an appropriate filtration on cogenerators.

Now, recall that $\hat{K}_G = \hat{H}(G) \otimes_{H(G)} K_G$.
We can thus extend the previous map to a map of Hopf comodules:
\begin{equation}\label{eq:definition of FrM^W}
  \Fr_{M}^{W,\mathrm{alg}} \coloneqq A \otimes_{H(\BG)} \hat{K}_G = (A\otimes \hat H(G), d) \to \modo_{A_G}(\Fr_{M}).
\end{equation}
The source of that map is a pushout.
The map is defined on the $A$ factor by
\begin{equation}\label{eq:a-mod-ag-fr-zigzag}
  A \xrightarrow{\sim} \OmPA(M) \to \OmPA(\Fr_M) \hookrightarrow \modo_{A_G}(\Fr_M),
\end{equation}
and on the $\hat{K}_G$ factor, it is defined using the map $\hat{H}(G) \to A_G$ defined above, and the unit $K_G \hookrightarrow \modo_{A_G}(*) \to \modo_{A_G}(\Fr_M)$ which extends $H(G) \to A_G$ into an $H(\BG)$-dgca map, where $H(\BG)$ acts on $\modo_{A_G}(*)$ trivially (i.e., $p_i \mapsto 0$ and $E \mapsto 0$).

\subsection{Graphical model for the fiber-wise little discs operad}
\label{sec:graph-model-fiber}

We now define the dgca model for $\FM_{n}^{M}$.
It works as follows.
If $Y$ is a $G$-space and $X$ a right $G$-space with a free $G$-action, then
\begin{equation}
  \begin{tikzcd}
    X \times_G Y \ar[r] \ar[d] \ar[dr, phantom, "\lrcorner" near start] & X/G \arrow{d} \\
    Y\sslash G \arrow{r} & \BG
  \end{tikzcd}
\end{equation}
is a homotopy pullback square.
We can apply this to $X =\Fr_{M}$ and $Y =  \FM_{n}$ to obtain that $\FM_{n}^{M}(r) = \Fr_{M} \times_{G} \FM_{n}(r)$ fits in a homotopy Cartesian square:
\begin{equation}
  \begin{tikzcd}
    \FM_{n}^{M}(r) \ar[r] \ar[d] \ar[dr, phantom, "\lrcorner" near start] &  \FM_{n}(r)\sslash G \ar[d] \\
    M=\Fr_{M}/G  \ar[r] & \BG
  \end{tikzcd}\, .
\end{equation}
Recall from Equation~\eqref{eq:3} that $A = \Graphs_M^{\mathrm{tree}}(1)$.
Using the pullback-to-pushout lemma \cite[Theorem~2.4]{Hess2007}, \cite[Proposition~15.8]{FHT2015}, we then obtain that:
\begin{equation}
  \stG_{n}^{M}(r) \coloneqq A \otimes_{H(\BG)} \pdu\BGraphs_{n}^m(r)
\end{equation}
is a dgca model of $\FM_{n}^{M}(r)$.

Let us now give a more concrete description of $\stG_{n}^{M}$.
This dgca is isomorphic to:
\begin{equation}
  \stG_{n}^{M}(r) \cong (A \otimes \stG_{n}(r), d_{A} + \delta_{\mathrm{contr}} + ET \cdot),
\end{equation}
where $d_{A}$ is the differential on $A$, $\delta_{\mathrm{contr}}$ is the differential on $\stG_{n}(r)$, $E \in A$ is the Euler class, and $T \cdot$ is the action of the tadpole graph on $\stG_{n}$:
\begin{equation}
  T=
  \begin{tikzpicture}
    \node (v0) at (0.5,1) [int] {};
    \draw (v0) to [out=45,in=135,loop] ();
  \end{tikzpicture}
  .
\end{equation}
Concretely, this last part of the differential is the sum over all possible ways of removing a non-tadpole edge from the graph and multiplying the element of $A$ by the Euler class.
The product is the product of $A$ and the product of $\stG_{n}$.
The cooperad structure is given by the morphisms of $A$-modules in CDGAs:
\begin{equation}\label{eq:cooperad-fiberwise-graphs}
  \stG_{n}^{M}(r+s-1) \to \stG_{n}^{M}(r) \otimes_{A} \stG_{n}^{M}(s)
\end{equation}
that are defined using the cooperad structure of $\stG_{n}$.
Note that this structure is compatible with the action of the dgca $A = \Graphs_{M}^{\mathrm{tree}}(1)$ as we have modded out by $\le 1$-valent vertices in $\Graphs_{n}$.

We have a direct morphism of Hopf cooperads
\[ \omega\colon \stG_{n}^{M} \to \OmPA(\FM_{n}^{M}), \]
defined as follows.
Recall from Section \ref{sec:graph-models-fm_m}, the form $\varphi \in \OmPA^{n-1}(\FM_M(2))$ and the form $\eta \in \OmPA^{n-1}(M)$.
Note that $\partial \FM_M(2) = \FM_n^M(2) \xrightarrow{\pi} M$ is a sphere bundle on $M$ and that the restriction of $\varphi$ to that bundle is a fiberwise volume form.
Then we define
\[ \rho \coloneqq \varphi|_{\partial \FM_M(2)} - \pi^*\eta \in \OmPA^{n-1}(\FM_n^M(2)). \]
Moreover, given $a \in A$, we can consider its image in $\OmPA(M)$ under the quasi-isomorphism of Equation~\eqref{eq:3}, then pull it back to $\FM_{n}^{M}$ using the projection.
By abuse of notation, we still denote by $a$ this element in $\OmPA(\FM_{n}^{M})$.
Then for an element $a \otimes \Gamma \in \stG_{n}^{M}(r)$, such that $\Gamma$ has $s$ internal vertices and no tadpoles, we define $\omega(\Gamma) \in \OmPA(\FM_{n}^{M}(r))$ by:
\begin{equation}\label{eq:form given by fiberwise-graphs}
  \omega(\Gamma) \coloneqq a \wedge \int_{\FM_{n}^{M}(r+s) \to \FM_{n}^M(r)} \bigwedge_{(i \neq j) \in E_{\Gamma}} p_{ij}^{*}\rho.
\end{equation}
If $\Gamma$ has tadpoles then $\omega(\Gamma) \coloneqq 0$.

\begin{rem}\label{rem:r-0-1}
  If $r \le 1$ then the formula above would not give the correct answer because the bundle $\FM^{M}_{n}(r+s) \to \FM^{M}_{n}(r)$ is of the wrong rank (essentially because $\dim \FM_{n}(r) = 0$ and not $nr-n-1$ in these cases).
  In the case $r=0$, $\FM_{n}^{M}(0) = M$, and $\stG_{n}^{M}(0) \cong A$, so we can define $\omega$ by \eqref{eq:3}.
  In the case $r=1$, we have $\stG_{n}^{M}(1)= (A \otimes \Graphs_n(1), d)$ and $\FM_{n}^{M}(1)=M$, so we can also take equation \eqref{eq:3} as a definition of the map $\omega$, tensored with the quasi-isomorphism $\Graphs_n(1) \to \R$.
\end{rem}

To explain where the map $\omega$ is compatible with the differentials, recall that the differential in $\stG_{n}^{M}$ it is defined to reflect how the Stokes formula interacts with the boundary strata of our compactifications and with the projections that forget some points. More precisely, there are three possible boundary cases: internal points might collide with one another, they might collide with a given external point or internal points might escape to infinity.
In $\Graphs_{n}$, the only nonzero contribution is when a single edge is contracted.
However, in $\Graphs_n^M$, the situation is in principle more complicated as the differential may contain more terms.
Given a graph $\Gamma \in \Graphs_n^M(r)$, the differential has in principle additional terms for every subgraph that contains at most one external vertex.
In that term, the subgraph in question is contracted.
If we denote by $\gamma \in \GC_{n}$ the subgraph with decorations removed (see Section~\ref{sec:gcn}), then the coefficient of that contraction is given by the ``partition function'':
\begin{equation}\label{eq:explicit-z}
  \begin{aligned}
    z : \GC_{n} & \to \Omega^{*+1}_{PA}(\FM_{n}^{M}(1)) = \Omega^{*+1}_{PA}(M), \\
    \gamma      & \mapsto
    \begin{cases}
      E,                                                                               & \text{if } \gamma = T \text{ is a tadpole}; \\
      \int_{\FM_{n}^{M}(k) \to M} \bigwedge_{(i,j) \in E_{\gamma}} p_{ij}^{*} \varphi, & \text{ otherwise}.
    \end{cases}
  \end{aligned}
\end{equation}

There could a priori be a problem in the definition of $\stG_{n}^{M}$, as the form $z(\gamma)$ might not be in the image of $A \to \OmPA(M)$.
However, thanks to the following lemma, this cannot happen, and we recover the behavior of the differential above:
\begin{lemma}\label{lem:z}
  The partition function $z$ vanishes on all graphs other than tadpoles.
\end{lemma}
\begin{proof}
  The argument is similar to the one of~\cite[Lemma~9.4.3]{LambrechtsVolic2014}.
  Let $\gamma$ be a graph with only internal vertices, different from a tadpole, and consider $z(\gamma)$.
  If the graph $\gamma$ has univalent or isolated vertices, then $z(\gamma) = 0$ by a simple dimension argument.
  If $\gamma$ has a bivalent vertex, then Kontsevich's trick~\cite[Lemma~2.1]{Kontsevich1994} shows that $z(\gamma)$ vanishes by a symmetry argument (the graph
  \begin{tikzpicture}
    \node (v0) at (0.3,1) [int] {};
    \node (v1) at (-0.3,1) [int] {};
    \draw (v1) to[out=30, in=150] (v0) ;
    \draw (v1) to[out=-30, in=-150] (v0) ;
  \end{tikzpicture}
  is not covered by Kontsevich's trick but it vanishes by symmetry reasons).
  Finally, if $n \geq 3$ and $\gamma$ only has vertices that are at least trivalent, then $z(\gamma)$ vanishes by a degree counting argument.

  For $n = 2$, then one must use a more sophisticated proof technique found in~\cite[Lemma~6.4, Section~6.6]{Kontsevich2003}.
  We can roughly paraphrase it as follows.
  In this case, Kontsevich proves that the form $\bigwedge_{i=1}^{2r} d \operatorname{arg}(z_{i})$ equals $\bigwedge_{i=1}^{2r} d \log|z_{i}|$.
  Kontsevich then chooses a compactification such that the boundary is a divisor with normal crossings; using some complex analysis and the Stokes formula, he concludes that the integral vanishes.
\end{proof}

\begin{rem}
  This lemma can be compared with~\cite[Theorem~7.1]{KhoroshkinWillwacher2017} (cf. Theorem~\ref{thm:MCform}), which is a much more general statement where the space $M$ is roughly speaking replaced by the classifying space $\BO(n)$.
  The result of the theorem is that $m$ is gauge equivalent (but not necessarily equal on the nose) to the trivial partition function.
  Our argument is much simpler here due to the fact that $\dim M = n$ so we may use the degree counting argument, whereas $\BO(n)$ is infinite-dimensional.
\end{rem}

\begin{thm}
  The morphism $\omega$ above defines a quasi-isomorphism of homotopy Hopf cooperads under the dgca map $A\to \OmPA(M)$:
  \[ \omega : \stG_{n}^{M} \xrightarrow{\sim} \OmPA(\FM_{n}^{M}). \]
\end{thm}
\begin{proof}
  Checking that $\omega$ commutes with all the structures involved is done by arguments very similar to the ones found in~\cite{LambrechtsVolic2014,CamposWillwacher2016,Idrissi2016}, using theorems of~\cite{HardtLambrechtsTurchinVolic2011}.
  It is compatible with:
  \begin{itemize}
    \item the differential, by the Stokes formula~\cite[Proposition~8.3]{HardtLambrechtsTurchinVolic2011} and the additivity of integration along fibers~\cite[Proposition~8.11]{HardtLambrechtsTurchinVolic2011}; in the Stokes formula for integrations along fibers $F$,
          \begin{equation}
            d \int_F \omega = \int_F d\omega + \int_{\partial F} \omega,
          \end{equation}
          the first summand corresponds to the tadpole action, and the second summand is the integral along the fiberwise boundary, which splits as several summands that exactly correspond to edge contractions (see the analogous argument in \cite[Proposition 9.4.1]{LambrechtsVolic2014} and note that we use Lemma~\ref{lem:z});
    \item the products, by the multiplicative property of integration along fibers \cite[Proposition~8.15]{HardtLambrechtsTurchinVolic2011};
    \item the cooperad structure, by an immediate check on the dgca generators;
    \item the identification of internal components $\gamma$ with $0$, by the multiplicative property of integral along fibers and the double pushforward formula~\cite[Proposition~8.13]{HardtLambrechtsTurchinVolic2011} (the argument is the same as the one of \cite[Lemma 9.3.7]{LambrechtsVolic2014} by a dimension argument).
  \end{itemize}
  Finally, it is a quasi-isomorphism by the fact that the model of the homotopy pullback is the homotopy pushout of the models (and since the maps we consider are (co)fibrations then we can remove the adjective ``homotopy'').
\end{proof}

\begin{thm}
  \label{thm:graphs-m-model}
  The symmetric sequence $\stG_{M}$ is a Hopf right $\stG_{n}^{M}$-comodule, and we have a quasi-isomorphism of homotopy Hopf right comultimodules:
  \[ (\omega,\omega) : (\stG_{M}, \stG_{n}^{M}) \xrightarrow{\sim} (\OmPA(\FM_{M}), \OmPA(\FM_{n}^{M})). \]
\end{thm}
\begin{proof}
  The proof is a direct extension of the proof of the main theorem of~\cite{CamposWillwacher2016} (see \cite[Proposition~13 and Lemma~16]{CamposWillwacher2016}).
  We can define maps
  \begin{equation}
    \label{eq:14}
    \circ_{i}^{*} : \stG_{M}(r+s-1) \to \stG_{M}(r) \otimes_{A} \stG_{n}^{M}(s)
  \end{equation}
  in a straightforward way, using subgraph contraction.
  The fact that these maps $\circ_i^*$ commute with $\omega$ is immediate from the definitions.
  Checking that these maps $\circ_i^*$ commute with the differential is a straightforward computation using the explicit description of the differential (compare with \cite[Lemma~16]{CamposWillwacher2016}).
  Note that while tadpoles are removed from the final version of $\Graphs_M$ in \cite{CamposWillwacher2016}, the compatibility with the differential on tadpoles is proved in \cite[Proposition~14]{CamposWillwacher2016}.
\end{proof}

\section{The framed configuration module}
\label{sec:fram-conf-module}

In this section we now give the model for $\FM_{M}^{\fr}$ seen as a right $\FM_{n}^{\fr}$ module.
We first give a general ``framing'' construction that we will then specialize to the case $\FM_{M}^{\fr}$.

\subsection{Constructions for right multimodules}
\label{sec:constr-right-mult}
Let $X$ be a topological space and $\op P$ be an operad in spaces over $X$. Let $\op M$ be a right $\op P$-multimodule.
Suppose that $f:Y\to X$ is some map of topological spaces.
Then we may define the pullback $Y\times_f \op P$ such that $(Y\times_f \op P)(r) = Y\times_f \op P(r)$. It is an operad in spaces over $Y$.
Similarly, the pullback $Y\times_f \op M$, defined such that $(Y\times_f \op M)(r):= Y^r \times_{f^r} \op M(r)$, is a right $Y \times_f \op P$-multimodule.
Let $\phi: \op Q\to \op P$ be a map of operads in spaces over $X$. Then $\op M$ can be naturally made into a right $\op Q$-multimodule, which we shall denote by $\phi^*\op M$ (accepting a slight clash in notation).

Next, suppose that $\op P=X\times \op R$, where $\op R$ is an ordinary topological operad.
Then $\op M$ can be made into a right $\op R$-module $\op M|_{\op R}$, via $\op M|_{\op R} \circ \op R = \op M \circ_X \op P \to \op M$.

\subsubsection{Framing construction}

Consider now the following input data:
\begin{enumerate}
  \item An operad $\op P$ in spaces over $X$ as above;
  \item A right $\op P$-multimodule $\op M$;
  \item A bundle over $X$, $f: F\to X$;
  \item A trivializing morphism:
        \beq{equ:trivmorphism}
        \phi: F\times \op R\to F\times_f \op P
        \eeq
        of operads in spaces over $F$, where $\op R$ is an ordinary operad.
\end{enumerate}

To these input data we associate the right operadic $\op R$-module
\newcommand{\Fra}{\mathrm{Fra}}
\[
  \Fra'_{\op P,\op M,f,\phi}
  :=
  (\phi^* (F\times_f \op M) )|_{\op R}.
\]

It is clear that the construction is functorial in the input data.
Let $G$ be a topological group.
If, in addition, $\op R$ is an operad in $G$-spaces (i.e.\ each space $\op R(r)$ carries an action of $G$ and all operadic maps are $G$-equivariant in an appropriate sense), and if the bundle $F$ carries a fiberwise $G$-action such that $\phi$ is $G$-equivariant, then $\Fra'_{\op P,\op M,f,\phi}$ is a right $\op R$-module in $G$-spaces. By this we mean (abusively) that $\Fra'_{\op P,\op M,f,\phi}(r)$ carries an action of $G^r$ such that the composition morphisms are $G$-equivariant in a natural sense.
This implies in particular that $\Fra'_{\op P,\op M,f,\phi}$ carries a right action of the framed operad $\op R\circ G$.

Now, suppose that the above input data is such that the bundle $F\to X$ is a principal $G$ bundle and $\op P = F\times_G \op R$.
Then there is a natural trivializing morphism $\phi: F\times \op R\to F\times_f \op P$, defined by:
\[
  \phi(a,b) = (a, (a\times_G b)).
\]
It is furthermore $G$-equivariant.
In this special case, we denote the right $\op R$-multimodule $\Fra'_{\op P,\op M,f,\phi}$ alternatively by
\[
  \Fra_{\op R, F, \op M} \coloneqq \Fra'_{F \times_{G} \op R, \op M, f, \phi}.
\]

\begin{ex}
  The example to keep in mind is $X = M$, $\op R = \FM_{n}$ and $\op M = \FM_{M}$, $F=\Fr_{M}$.
  Then we have that $\op P = F \times_{G} \op R = \FM_{n}^{M}$, and $\Fra_{\op R, F, \op M} = \FM_{M}^{\fr}$.
\end{ex}

\subsubsection{Functoriality} \label{sec:Frafunctoriality}

Let us observe that the constructions $\Fra'_{\op P,\op M, f,\phi}$ and $\Fra_{\op R, F, \op M}$ depend functorially on the data.
Indeed, suppose that we have two tuples $(\op P,\op M, f,\phi)$, $(\op P',\op M', f',\phi')$ as above. Suppose moreover that we have morphisms:
\begin{align*}
  \alpha: \op P & \to \op P'
                &
  \beta: \op R  & \to \op R'
                &
  \gamma: \op M & \to \op M',
\end{align*}
and a morphism of bundles $\delta:F\to F'$ from the bundle $f:F\to X$ to $f':F'\to X$.
Finally, suppose that our morphisms respect the naturally given structure on objects.
In particular, they make the following diagrams commute:
\beq{equ:functoriality_tocheck}
\begin{tikzcd}
  \op M \ar{d}{\gamma} & \op P \ar{d}{\alpha} \ar[dashed]{l} \\
  \op M' & \op P' \ar[dashed]{l}
\end{tikzcd}
\quad\quad
\begin{tikzcd}
  F \times_f  \op P \ar{d}{ \delta \times_f\alpha} &  F \times \op R \ar{d}{\delta\times\beta } \ar{l}{\phi} \\
  F' \times_{f'}  \op P' &  F'  \times \op R' \ar{l}{\phi'}.
\end{tikzcd}
\eeq
In this diagram, we have indicated operadic (right) actions by dashed arrows.
It is then clear that the maps given provide a morphism of operads and their right modules:
\[
  \begin{tikzcd}
    \Fra'_{\op P,\op M,f,\phi}  \ar{d}{\delta\times_X\gamma } & \op R \ar{d}{\beta} \ar[dashed]{l} \\
    \Fra'_{\op P',\op M',f',\phi'}  & \op R' \ar[dashed]{l}.
  \end{tikzcd}
\]
Furthermore, if the maps $\beta$ and $\delta$ above are compatible with the $G$-action, then the operads $\op R$ and $\op R'$ on the right-hand side of the above diagram may be replaced by their $G$-framed versions.
If $\alpha, \beta, \gamma, \delta$ are weak equivalences, then so is the induced map of operadic right modules above.
(Note that we have used that $f, f'$ are fiber bundles and therefore fibrations.)

Let us also state a slightly relaxed version of the above functoriality result.
Suppose next that we have maps $\beta, \gamma, \delta$ as above, and in addition two homotopic maps of operads over $X$:
\[
  \begin{tikzcd}
    \op P \ar[shift left]{r}{\alpha_0} \ar[shift right, swap]{r}{\alpha_1} & \op P',
  \end{tikzcd}
\]
such that the following diagrams commute:
\beq{equ:functoriality_tocheck2}
\begin{tikzcd}
  \op M \ar{d}{\gamma} & \op P \ar{d}{\alpha_0} \ar[dashed]{l} \\
  \op M' & \op P' \ar[dashed]{l}
\end{tikzcd}
\quad\quad
\begin{tikzcd}
  F \times_f \op P  \ar{d}{\delta\times_f \alpha_1} &F  \times \op R \ar{d}{\delta\times \beta} \ar{l}{\phi}\\
  F' \times_{f'} \op P'  & F' \times \op R' \ar{l}{\phi'}.
\end{tikzcd}
\eeq

We claim that still we have a (homotopy) morphism $\Fra'_{\op P,\op M,f,\phi}\to  \Fra'_{\op P',\op M',f',\phi'}$.
Concretely, suppose that the homotopy is realized by a path in the mapping space:
\[
  \alpha: I \times \op P \to \op P',
\]
with $I=[0,1]$, whose endpoints agree with $\alpha_0, \alpha_1$ respectively.
Let us define the bundle $f_I: F_I:= I\times F\to X$ by trivial extension of $F$.
We will also define the trivializing morphism (of operads over $F_I$)
\[
  \phi_I :  F_I \times \op R \to  F_I \times_{f_I} \op P'
\]
as the composition
\[
  F_I \times \op R \cong  I  \times F\times \op R
  \xrightarrow{\mathit{id}_I \times \phi}
  I\times F \times_f \op P \to  (I\times F)\times_{f_I} (I\times \op P)
  \xrightarrow{\mathit{id} \times_{f_I} \alpha}
  F_I  \times_{f_I}\op P'.
\]
Here we used the diagonal $I\to I\times I$ for the middle arrow.
Furthermore, let us define the map
\[
  \tilde \phi :  F \times\op R \to   F \times_{f}\op P'
\]
as the composition
\[
  F \times \op R
  \xrightarrow{\phi} F \times_f \op P
  \xrightarrow{\mathit{id} \times_f \alpha_1} F \times_f \op P'.
\]

We then build the following zigzag:
\beq{equ:htpydiagram}
\begin{tikzcd}
  \Fra'_{\op P, \op M, f, \phi} \ar{d}{\rho_0} & \op R \ar[dashed]{l} \ar{d}{=}\\
  \Fra'_{\op P', \op M', f_I, \phi_I} & \op R \ar[dashed]{l} \\
  \Fra'_{\op P', \op M', f, \tilde \phi} \ar{d}{\rho_2} \ar{u}{\rho_1}[swap]{\sim}& \op R \ar[dashed]{l} \ar{d}{\beta} \ar{u}{=} \\
  \Fra'_{\op P', \op M', f', \phi'} & \op R' \ar[dashed]{l}
\end{tikzcd}.
\eeq
Let us explain the construction of the vertical maps, which are all obtained from maps on the input data of $\Fra'$ and functoriality.
The top vertical map on the left $\rho_0$ is obtained by the maps $\alpha_0:\op P \to \op P'$, $\gamma:\op M\to \op M'$ and the map $\iota_0:F\to F_I = I\times F$ sending an element $u\in F$ to $(0,u)$.
We use here that the following diagrams commute, which is evident by construction:
\[
  \begin{tikzcd}
    \op M \ar{d}{\gamma} & \op P \ar{d}{\alpha_0} \ar[dashed]{l} \\
    \op M' & \op P' \ar[dashed]{l}
  \end{tikzcd}
  \quad\quad
  \begin{tikzcd}
    F\times_f\op P  \ar{d}{\iota_0 \times_f \alpha_0} &   F\times\op R \ar{d}{ \iota_0 \times\mathit{id}}  \ar{l}{\phi} \\
    F_I\times_{f_I} \op P' &  F_I \times \op R \ar{l}{\phi_I} .
  \end{tikzcd}
\]
Similarly, the map $\rho_1$ is induced by the maps $\iota_1:F\to F_I = I\times F$ sending an element $u\in F$ to $(1,u)$.
One readily checks that the relevant diagrams commute:
\[
  \begin{tikzcd}
    \op M' \ar{d}{=} & \op P' \ar{d}{=} \ar[dashed]{l} \\
    \op M' & \op P' \ar[dashed]{l}
  \end{tikzcd}
  \quad\quad
  \begin{tikzcd}
    F\times_f \op P' \ar{d}{\iota_1\times_f \mathit{id}} &  F\times \op R \ar{d}{ \iota_1 \times\mathit{id}}  \ar{l}{\tilde \phi} \\
    F_I \times_{f_I} \op P' &  F_I\times \op R \ar{l}{\phi_I} .
  \end{tikzcd}
\]
Finally, the map $\rho_2$ is defined by functoriality of $\Fra'$ and the maps on input data $\beta:\op R\to \op R'$ and $\delta:F\to F'$. One again checks that the relevant diagrams
\[
  \begin{tikzcd}
    \op M' \ar{d}{=} & \op P' \ar{d}{=} \ar[dashed]{l} \\
    \op M' & \op P' \ar[dashed]{l}
  \end{tikzcd}
  \quad\quad
  \begin{tikzcd}
    F\times_f\op P' \ar{d}{\delta\times_f \mathit{id}} & F \times \op R \ar{d}{ \delta \times \beta}  \ar{l}{\tilde \phi} \\
    F' \times_{f'} \op P' & F' \times \op R' \ar{l}{\phi'} .
  \end{tikzcd}
\]
commute. Overall, we have constructed a zigzag (the LHS of Equation~\eqref{equ:htpydiagram}), in which the only arrow pointing in the upward direction is a weak equivalence. In other words, we have constructed a homotopy morphism $\Fra'_{\op P,\op M,f,\phi}\to  \Fra'_{\op P',\op M',f',\phi'}$, as desired. If the maps on the input data are weak equivalences, then so are the maps in our zigzag.

Furthermore, if all data respect the $G$-actions, then we may pass to modules over the $G$-framed operads.

The constructions above readily extend to homotopy operads and modules, and dualize to (homotopy) Hopf cooperads and comodules.

\subsection{The framed configuration module $\FM_{M}^{\fr}$}

Using the above general construction we can now define the framed configuration module of $M$ as
\[
  \FFM_{M} \coloneqq \Fra_{\FM_n, \Fr_M, \FM_M}.
\]
More concretely, we have
\begin{equation}
  \FFM_{M}(r) = (\FM_{M} \circ_{M} \Fr_{M})(r) = \FM_{M}(r) \times_{M^{r}} \Fr_{M}^{r},
\end{equation}
where the fiber product is defined using the maps $p_1,\dots,p_r$ of Equation~\eqref{equ:projections}.

By construction, $\FFM_M$ carries a natural operadic right action of the framed Fulton--MacPherson operad $\FFM_n$.
In particular, $\FFM_M(r)$ carries a natural $G^r$-action, i.e., we have compatible maps:
\begin{equation}
  \label{equ:gaction}
  \FFM_M \circ G \to \FFM_M,
\end{equation}
and an operadic right $\FM_{n}$-module structure.
The composition morphisms implicit in the construction $\Fra_{\FM_n, \Fr_M, \FM_M}$ are given explicitly as part of the following composition:
\begin{multline}
  \label{eq:ffm-m-comodule}
  \FFM_M \circ \FM_n
  =
  \FM_M \circ_M \Fr_M \circ \FM_n
  \xrightarrow{\Delta}
  \FM_M \circ_M \Fr_M \circ \FM_n \circ_M \Fr_M
  \xrightarrow{\pi}
  \\
  \xrightarrow{\pi}
  \FM_M \circ_M \FM_n^M \circ_M \Fr_M
  \to
  \FM_M \circ_M \Fr_M
  = \FFM_M,
\end{multline}
where $\Delta$ is the diagonal map of $\Fr_{M}$ and $\pi : \Fr_{M} \times \FM_{n} \to \FM_{n}^{M}$ is the projection.

\subsection{The graphical model $\stG_{M}^{\fr}$}
Let us first define our (at this point of the paper tentative) graphical model for $\FFM_{M}$.
In the next subsection, we will relate this graphical model to the forms on (a version of) $\FFM_{M}$ by an explicit zigzag of quasi-isomorphisms of homotopy Hopf comodules.

By definition, the space $\FFM_{M}(r)$ fits into the following pullback diagram:
\begin{equation}
  \begin{tikzcd}
    \FFM_M(r) \ar[r] \ar[d] \ar[dr, phantom, "\lrcorner" near start] & \FM_M(r) \arrow{d} \\
    (\Fr_M)^{\times r} \arrow{r} & M^{\times r}
  \end{tikzcd}
\end{equation}

Recall the model $\Fr_{M}^{W,\mathrm{alg}} = (A \otimes \hat H(G), d)$ of $\Fr_{M}$ obtained in Section~\ref{sec:model-frame-bundle}.
We can take as an algebraic model for $\FFM_{M}(r)$ the following dgca:
\begin{equation}
  \begin{aligned}
    \stG_M^{\fr}(r) & \coloneqq (A \otimes \hat H(G), d)^{\otimes r} \otimes_{A^{\otimes r}} \stG_M(r) \\
                    & \cong \bigl( \stG_M\circ_{A}(A\otimes \hat H(G), d) \bigr)(r)                    \\
                    & \cong \bigl( \stG_M\circ \hat H(G)(r), d \bigr).
  \end{aligned}
\end{equation}

Note that some care has to be taken.
One cannot immediately apply the pullback-to-pushout lemma, since the base $M^r$ is not necessarily simply connected.
We may however still conclude that $\stG_M^{\fr}(r)$ is a dgca model for $\FFM_M(r)$ since $\FFM_M(r)\to \FM_M(r)$ is a $G^r$-principal bundle.
Hence, one can explicitly write down a model given a model of the base as in \cite[Theorem 1, section 9.3]{GHV1976}, which in this case agrees with $\stG_M^{\fr}(r)$.

We can now define the $\stG_{n}$-comodule structure on $\stG_{M}^{\fr}$ by writing a diagram which is dual to Equation~\eqref{eq:ffm-m-comodule}:
\begin{multline}
  \stG_M^{\fr} = (\stG_M\circ \hat H(G), d)
  \to \\
  \to
  \stG_M \circ_A (A\otimes \stG_n, d) \circ_A (A\otimes \hat H(G), d)
  \xrightarrow{\pi^{*}}
  \\
  \xrightarrow{\pi^{*}}
  \stG_M \circ_A((A \otimes \hat H(G), d) \otimes \stG_n) \circ_A (A\otimes \hat H(G), d)
  \xrightarrow{\operatorname{mult}_{\Fr_{M}^{W,\mathrm{alg}}}}
  \\
  \xrightarrow{\operatorname{mult}_{\Fr_{M}^{W,\mathrm{alg}}}}
  (\stG_M\circ \hat H(G), d) \circ \stG_n = \stG_{M}^{\fr} \circ \stG_{n}.
\end{multline}
Here, the map $\operatorname{mult}_{\Fr_{M}^{\mathrm{alg}}}$ is the product of the dgca $\Fr_{M}^{\mathrm{alg}}$ (i.e., the dual of the diagonal map), and $\pi^{*} : \stG_{n}^{M} \to (A \otimes \hat H(G), d) \otimes \stG$ is the map defined by:
\begin{equation}
  \label{eq:pi-star}
  \begin{aligned}
    \pi^{*} : (A \otimes \stG_{n}, d)
     & \to (A \otimes \hat H(G), d) \otimes \stG_{n} \\
    a \otimes \Gamma
     & \mapsto \sum a \otimes b' \otimes \Gamma'',
  \end{aligned}
\end{equation}
where we use Sweedler notation $\Gamma\mapsto \sum b' \otimes \Gamma''$ to describe the $\hat H(G)$-coaction on $\stG_n$ of Section \ref{sec:equiv-graph-models}.
It can be seen from the discussion below
that $\pi^*$ is in fact the model of the projection $\pi : \Fr_{M} \times \FM_{n} \to \FM_{n}^{M}$.

In addition to the coaction of the commutative Hopf algebra $\hat H(G)$ on $\Graphs_n$, we have compatible coactions of $\hat H(G)^r$ on $\stG_M^{\fr}(r)$.
We can therefore pass to the framed Hopf cooperad $\Graphs_n^{\fr}=\Graphs_n\circ \hat H(G)$, which inherits a coaction of $\stG_M^{\fr}$.

In conclusion, our (tentative) graphical model for the pair ($\FM_M^{\fr}$, $\FM_n^{\fr}$) is the pair ($\stG_M^{\fr}$, $\Graphs_n^{\fr}$).
Note that our graphical model is an honest Hopf comodule/cooperad, not just up to homotopy.

\subsection{The zigzag}
\label{sec:zigzag}
Our next goal is to prove that the tentative model built in the previous subsection is indeed a model of the pair $(\FM_M^{\fr}, \FM_n^{\fr})$.
We will do this by going through the steps of the construction $\FM_M^{\fr} = \Fra_{\FM_n,\Fr_M,\FM_M}$ from Section~\ref{sec:constr-right-mult}.
Recall that the input of that construction is the data of:
\begin{enumerate}
  \item The operad $\FM_n$ in $G=\SO(n)$-spaces;
  \item The principal $G$-bundle $\Fr_M\to M$;
  \item The right $\FM_n^M=\Fr_M\times_G \FM_n$-multimodule $\FM_M$;
  \item The trivializing morphism $\varphi : \Fr_{M} \times \FM_{n} \to \Fr_{M} \times_M \FM_{n}^{M}$.
\end{enumerate}

We will use the following combinatorial models for the above objects.
Let us first describe the first three:
\begin{enumerate}
  \item We will use two different models for the operad $\FM_n$ in $G$-spaces.

        First, recall that $\hat{B}_{n} = (\hat H(G)\otimes \BGraphs_n^m, d)$ defined in \eqref{eq:definition of Bn} is a complete Hopf cooperad with an $\hat H(G)$-coaction.
        This complete Hopf cooperad is a model of $\FM_n$ (or rather the homotopy equivalent realization of $WG\times G^\bullet \times \FM_n$) via the map:
        \[
          \hat{B}_{n} \to \Tot(\modo_{A_G}(\widehat{\FM}_{n})).
        \]
        We refer to \cite{KhoroshkinWillwacher2017} or Section \ref{sec:equiv-graph-models}.
        The map is a quasi-isomorphism of complete cooperads in homotopy $A_G$-comodules (where $A_G = W\OmPA(G)$).

        Second, we may take as a another model the Hopf cooperad $\Graphs_n$ with the $\hat H(G)$-coaction of Section \ref{sec:equiv-graph-models}.
        In contrast to the former model, this is an honest dg Hopf cooperad in $\hat H(G)$-comodules.
        There is a comparison quasi-isomorphism \eqref{eq:zigzag-bn}
        \[
          \Graphs_n\xrightarrow{\sim}  \hat H(G)\otimes \BGraphs_n^m
        \]
        by taking the coaction.

  \item For the principal bundle $\Fr_M$, we will use the model $\Fr_{M}^{W,\mathrm{alg}} = (A\otimes \hat H(G),d)$ defined in Equation \eqref{eq:definition of FrM^W}.

  \item For $\FM_M$ as right $\FM_n^M$-module, we use the model $\Graphs_M$ from Section \ref{sec:graph-models-fm_m}, coacted upon by $A \otimes \Graphs_n$.
\end{enumerate}

\begin{convention}\label{conv:W}
  In what follows, when we take the $W$-resolution of an $\OmPA(G)$-comodule, it is not always necessarily clear what action of $G$ is resolved.
  We are going to put the $W$ inside the $\Omega$, with the understanding that the resolution takes place in the category of $\OmPA(G)$-comodules, and that the action refers to the object in front of the $W$.
  For example, for $\widehat{\FM}_{n} = G \times G^{\bullet} \times \FM_{n}$, the action is on the first $G$, so we are going to write
  \[ W \OmPA(\widehat{\FM}_{n}) \eqqcolon \OmPA(WG \times G^{\bullet} \times \FM_{n}) \]
  to indicate where $G$ acts in the resolution.
  All the other notations below will be defined analogously.
\end{convention}

The fourth step in the construction $\Fra_{\FM_n,\Fr_M,\FM_M}$ is to build a model of the trivializing morphism \eqref{equ:trivmorphism}, i.e.\ the map $\Fr_{M} \times \FM_{n} \to \Fr_{M} \times_M  \FM_{n}^{M}$.
This is provided by the following commutative diagram of homotopy Hopf cooperads
under $A\otimes \hat H(G)$.
\begin{equation}
  \begin{tikzcd}[column sep=small]
    (A\otimes \hat H(G) \otimes_A  \Graphs_n^M, d)  \ar{d}{\sim}  \ar{r}{\sim}&
    (A\otimes \hat H(G)  \otimes \Graphs_n,d) \ar{d}{\sim}
    \\
    (A\otimes \hat H(G) \otimes_A  A \otimes \hat H(G) \otimes \BGraphs_n^m,d)  \ar{d}{\sim}  \ar{r}{\sim}&
    (A\otimes \hat H(G)  \otimes \hat H(G) \otimes \BGraphs_n^m,d) \ar{d}{\sim}
    \\
    \Tot \OmPA( \Fr_M^W \times_M \Fr_M^W \times G^\bullet\times \FM_n) \ar{r}{\sim} &
    \Tot \OmPA( \Fr_M^W \times WG\times G^\bullet\times \FM_n).
  \end{tikzcd}
\end{equation}
Note that in this diagram, the first two rows contain strict Hopf cooperads, while in the last row the homotopy Hopf cooperad structure is given by the usual construction, such as in Equation~\eqref{eq:omegaishomotopycoop}

The diagram also clearly respects the $A_G$-coaction (where $A_G = W\OmPA(G)$), where $G$ acts on $\Fr_{M}$ using the action defined in Section~\ref{sec:constr-right-mult}.

The next complication is that in the above diagram we used $\Fr_M^W \times G^\bullet\times \FM_n$ as our replacement of $\FM_n^M$, while in the model of the action on $\FM_M$, we used $\FM_n^M$.
Thus, we would in principle have to check the commutativity of the diagram of homotopy Hopf cooperads under $A$ given by:
\beq{equ:AGraphshoCD}
\begin{tikzcd}
  \Graphs_n^M=(A\otimes\Graphs_n,d) \ar{r}{\sim} \ar{d}{\sim}&
  (A \otimes \hat H(G) \otimes \BGraphs_n^m,d) \ar{d}{\sim}
  \\
  \OmPA(\FM_n^M) \ar{r}{\sim} &
  \Tot \OmPA(\Fr_M^W \times G^\bullet\times \FM_n).
\end{tikzcd}
\eeq
Unfortunately, this diagram does not commute. However, we will check below in Proposition \ref{prop:bgraphsdiag} and Corollary \ref{cor:bgraphsdiag} that it commutes up to homotopy.
This commutativity up to homotopy will be enough for our purposes, though it adds additional complications.
More concretely, we will dualize the argument of Section \ref{sec:Frafunctoriality} to create our final zigzag of quasi-isomorphisms of homotopy Hopf cooperads and their homotopy right Hopf comodules (coactions being depicted as dashed arrows) as follows:
\beq{equ:cohtpydiagram}
\begin{tikzcd}
  \OmPA(\FM_M \circ_M \Fr_M^W)  \ar[dashed]{r} & \Tot \OmPA(WG\times G^\bullet \times \FM_n)   \\
  \Graphs_M  \circ_A \OmPA(\Fr_M^W)[t,dt]  \ar[dashed]{r} \ar{u}{\nu_0} \ar{d}{\nu_1} & \Tot \OmPA(WG\times G^\bullet \times \FM_n)  \ar{u}{=} \ar{d}{=} \\
  \Graphs_M \circ_A \OmPA(\Fr_M^W) \ar[dashed]{r} & \Tot \OmPA(WG\times G^\bullet \times \FM_n) \\
  \Graphs_M \circ \hat H(G)  \ar[dashed]{r} \ar{u}{\nu_2} & \Graphs_n. \ar{u}
\end{tikzcd}
\eeq

A few explanations are in order, since the notation is somewhat compressed.
The reader is advised to follow the diagram \eqref{equ:htpydiagram} in parallel, of which \eqref{equ:cohtpydiagram} is the reformulation in the dual and homotopy setting.
We now describe each line and each vertical arrow in the diagram.

\begin{itemize}
  \item In the first line, $\Tot \OmPA(WG\times G^\bullet \times \FM_n)$ is a homotopy Hopf cooperad.
        The underlying functor sends a tree $T$ to the dgca
        \[
          \Tot \OmPA(WG\times G^\bullet \times \FM_n(T)).
        \]
        In particular, mind that for each tree there are multiple factors of $\FM_n(r)$, one for each node, but only one factor $WG$ for the whole tree -- otherwise we would not know how to define the contraction and gluing morphisms properly.

        Moreover, $\OmPA(\FM_M \circ_M \Fr_M^W)$ is a homotopy right comodule over the aforementioned cooperad.
        Concretely, let $T$ be a marked tree as in \eqref{equ:examplemoduletree}, i.e.\ the marked vertex has children $T_{1},\, \dots,\, T_{r}$.
        The functor $\OmPA(\FM_{M} \circ_{M} \Fr_{M}^{W})$ maps $T$ to the dgca:
        \[
          \Tot \OmPA \bigl( \FM_M(r) \times_{M^r} (\Fr_M^W)^r  \times \prod_{j=1}^r (WG \times G^\bullet \times \FM_n(T_j)) \bigr).
        \]
        The contraction morphisms (comodule structure) are defined as pullbacks of the topological composition, in the natural way.
  \item In the second line, the Hopf cooperad $\Tot \OmPA(WG \times G^{\bullet} \times \FM_{n})$ is the same.
        The homotopy Hopf comodule $\Graphs_M  \circ_A \OmPA(\Fr_M^W)[t,dt]$ assigns to a tree $T$ as before the dgca:
        \[
          \Tot \Graphs_M(r) \otimes_{A^r} \OmPA \bigl( (\Fr_M^W\times I)^r \times \prod_{j=1}^r (WG\times G^\bullet \times \FM_n(T_j)) \bigr),
        \]
        where $I=[0,1]$ is again the interval.

        The contraction morphisms are defined (dually and) analogously to the action on the module in the second line of \eqref{equ:htpydiagram}.
        More concretely, for the contraction (or rather expansion) of a top edge, we first take the corresponding $\Graphs_n^M$-coaction on the factor $\Graphs_M$. The resulting element in $\Graphs_n^M$ is sent to $\Tot \OmPA(\Fr_M^W \times G^\bullet \times \FM_n)[t,dt]$ using the explicit morphism from Corollary \ref{cor:bgraphsdiag}.
        Finally, the result can be mapped naturally to $\Tot \OmPA(\Fr_M^W \times WG \times G^\bullet \times \FM_n)[t,dt]$ using the pullback of the $WG$-action on $\Fr_M^W$.
  \item The left vertical upwards arrow $\nu_0$ is induced by the map $\Graphs_M\to \OmPA(\FM_M)$ of Section \ref{sec:graph-model-fiber}, and by the restriction to the endpoints $t=0$ of the intervals $I$.
  \item In the third line, the Hopf cooperad is still the same, but we restrict the comodule to $t = 1$.
        This comodule is defined analogously to the one in the second line, except that in the (co)contraction morphisms one does not see the full homotopy from Proposition \ref{prop:bgraphsdiag}; we just see the map at the endpoint $t=1$ of the interval, i.e., the upper composition in \eqref{equ:AGraphshoCD}.
  \item Correspondingly, the map $\nu_1$ is merely the restriction to the endpoints $t=1$ of the intervals $I$.
  \item In the last line, we see the $\Graphs_{n}$-comodule $\Graphs_{M} \circ \hat{H}(G)$ from Section~\ref{sec:equiv-graph-models}.
  \item Finally, the map $\nu_{2}$ is defined using the morphism $(A \otimes \hat{H}(G),d) \to \OmPA(\Fr_{M}^{W})$ from Section~\ref{sec:model-frame-bundle}
\end{itemize}

The diagram \eqref{equ:cohtpydiagram} realizes our desired zigzag of morphisms of homotopy Hopf cooperads and comodules.
Additionally, it is not hard to check that all arrows in the diagram are quasi-isomorphisms. (This is due to the fact that $\Graphs_M(r)$ is free as an $A^r$-module.)
We thus get as an intermediary result:
\begin{prop}
  The Hopf right comodule $(\Graphs_{M}^{\fr}, \Graphs_{n})$ is a model for $(\FM_{M}^{\fr}, \FM_{n})$.
\end{prop}

Finally, note that all objects in the diagram have a homotopy $A_G \simeq \hat H(G)$-coaction, compatible with all structures and maps.
We then desire to apply the framing construction to pass from $\FM_n$ to the framed counterpart $\FFM_n=\FM_n\circ G$.
\begin{prop}
  Let $\op T$ be a an operad in $G$-spaces and $\op M$ be a right $(\op T \circ G)$-module.
  Let $H=H(G)$ be the Hopf coalgebra which modelling $G$. Suppose that $\op C$ is a Hopf cooperad in $H$-comodules which models $\op T$ and that $\op N$ is a homotopy $\op C$-comodule with a compatible $H$-coaction which models $\op M$.
  Then the natural coaction of $\op C \circ H$ on $\op N$ models the action of $(\op T \circ G)$ on its right module $\op M$.
\end{prop}

Note that this is, a priori, not completely obvious.
For example, we do not know how to apply the framing construction directly to homotopy (co)operads.

\begin{proof}
  Let us sketch the proof of~\cite[Theorem 5.5]{KhoroshkinWillwacher2017}, which can be adapted to the case of comodules immediately.
  In the same way as we constructed the first part of~\eqref{eq:zigzag-bn-omega}, we obtain a zigzag, where dashed arrows represent coactions, $K_G$ is the Koszul complex, and $\upsilon$ is the map from Proposition~\ref{prop:hopf-formal}:
  \begin{equation}
    \begin{tikzcd}
      H(G) \ar[d,dashed]
      & H(G) \ar[d,dashed] \ar[l, equal] \ar[r, "\upsilon"]
      & A_G \ar[d,dashed] \\
      K_G \otimes_{H(\BG)} \op C
      & \Tot \modo_{A_G}(\widehat{\op T}) \ar[l, "\sim"] \ar[r, "\sim"]
      & \modo_{A_G}(\op T)
    \end{tikzcd}
  \end{equation}

  After applying the Boardman--Vogt construction, and using the fact that $(-) \circ H(G)$ (resp., $(-) \circ A_G$) is exact, we get a zigzag:
  \begin{equation}
    (K_G \otimes_{H(\BG)} \op C) \circ H(G) \xleftarrow{\sim} W(K_G \otimes_{H(\BG)} \op C) \circ H(G) \xrightarrow{\sim} W(\modo_{A_G}(\op T)) \circ A_G.
  \end{equation}
  Moreover, we have a zigzag of quasi-isomorphisms of homotopy Hopf cooperads (cf.~\cite[Proposition 5.4]{KhoroshkinWillwacher2017}):
  \najibline{After sitting down I think I see what the first map is now. As far as I can tell, it's actually mostly straightforward but tedious to write down... Should we include the details?}

\ricardoline{I don't think it's very useful for the amount of space it would take. To understand this one needs to write this down oneself, I think.}
  \begin{equation}
    \modo_{A_G}(\op T \circ G) \xrightarrow{\sim} \OmPA(W_{G\text{-op}}(W_G(\op T)) \circ WG) \xleftarrow{\sim} \OmPA(W_{\text{op}}(\op T) \circ G) \xleftarrow{\sim} \OmPA(\op T \circ G),
  \end{equation}
  where $W_G$ is the Boardman--Vogt construction for $G$-spaces, $W_{G\text{-op}}$ for operads in $WG$-spaces, and $W_{\text{op}}$ for topological operads.
  If we replace $\op T$ by $\op M$ and $\op C$ by $\op N$ in the above zigzags, we also obtain a zigzag of quasi-isomorphisms in Hopf operadic comodules (under the corresponding Hopf cooperad of the above zigzag), proving the claim.
\end{proof}
Using this proposition, we can conclude:
\begin{thm}
  \label{thm:model-ffmm}
  The Hopf right comodule $(\stG_{M}^{\fr}, \stG_{n})$ is a model for $(\FFM_{M}, \FM_{n})$, i.e., there exists a zigzag of quasi-isomorphisms of (homotopy) Hopf cooperads/comodule between $(\stG_{M}^{\fr}, \stG_{n})$ and $(\OmPA(\FFM_{M}), \OmPA(\FM_{n}))$.

  Furthermore, the $\SO(n)$-actions on $(\FFM_{M}, \FM_{n})$ are modelled (in the same sense) by the natural $\hat H(\SO(n))$-actions on the pair $(\stG_{M}^{\fr}, \stG_{n})$.
  It follows that the pair $(\stG_{M}^{\fr}, \stG_{n}^{\fr})$ is a model of $(\FFM_{M}, \FM_{n}^{\fr})$.
\end{thm}

\section{Parallelized manifolds and the relation to earlier work}
\label{sec:frame-change}

In this section, we study how our models behave on parallelizable manifolds, i.e., manifolds whose tangent bundles are trivial.

\subsection{Choosing a framing}
\label{sec:choosing-framing}

In this subsection, we assume that $M$ is parallelized, i.e., it comes equipped with a chosen section of the frame bundle:
\[
  s: M\to \Fr_M.
\]
In this case we may define a right $\FM_n$-action on the (non-framed) configuration space $\FM_M$ directly.
It is shown in \cite{CamposWillwacher2016,Idrissi2016} that, provided proper choices are made in the construction, there is a natural $\Graphs_n$-coaction on our model $\Graphs_M$ for $\FM_M$.
This coaction models the $\FM_n$-action on $\FM_M$.
In this section, we shall elucidate this action and describe the relationship with the present work.

To this end, let us note that the action of $\FM_n$ on $\FM_M$ in the parallelized setting may be obtained from the action of $\FM_n$ on $\FM_n^{\fr}$ and from the section of the frame bundle $s$ as follows.
The space $\FM_M$ is the pullback:
\[
  \begin{tikzcd}[row sep = small, column sep = small]
    \FM_M(r) \ar{r} \ar{d} \ar[dr, phantom, "\lrcorner"] & \FM_M^{\fr}(r) \ar{d} \\
    M^r \ar{r}{s^r} & \Fr_M^r
  \end{tikzcd}.
\]
The action of $\FM_n$ may be recovered by functoriality of the pullback, as the dashed edge in the following map of pullback squares:
\[
  \begin{tikzcd}[cramped, sep=small]
    & \FM_M(r+s-1) \ar{rr} \ar{dd} & & \FM_M^{\fr}(r+s-1) \ar{dd} \\
    \FM_M(r) \times \FM_n(s) \ar[dashed]{ur} \ar{dd} \ar[crossing over]{rr} & & \FM_M^{\fr}(r) \times \FM_n(s) \ar{ur}{\circ_j} & \\
    & M^{r+s-1} \ar{rr} & & \Fr_M^{r+s-1} \\
    M^r \ar{rr}{s^{r}} \ar{ur}{\Delta}& & \Fr_M^r \ar{ur}{\Delta} \ar[crossing over, leftarrow]{uu}&
  \end{tikzcd}
\]
In the diagram, the map $\circ_j$ denotes the action of $\FM_n$ on $\FM_M^{\fr}$ on $\FM_M^{\fr}$.
The arrows labeled $\Delta$ are $(s-1)$-fold diagonals, applied to the $j$-th factor in the product.

Let us dualize the above diagram and constructions, using our model ($\Graphs_M^{\fr}$, $\Graphs_n$) for the framed configuration space comodule.
Let us assume that our section $s$ is modeled by the dgca morphism:
\beq{equ:sigma}
\sigma : (A\otimes \hat H(G), d) \to A
\eeq
From this data, we obtain a model for $\FM_n(r)$ as the pushout
\[
  \begin{tikzcd}
    \Graphs_M^{\fr}(r) \otimes_{ (A\otimes \hat H(G))^r } A^r \cong \Graphs_M(r)
    \ar[dr, phantom, "\lrcorner"]
    & \Graphs_M^{\fr}(r)  \ar{l} \\
    A^r \ar{u} & (A\otimes \hat H(G), d)^r \ar{l}{\sigma^r} \ar{u}
  \end{tikzcd}\, .
\]
Here we again cannot readily apply the pullback-to-pushout Lemma since the base is not simply connected. However, one nevertheless verifies explicitly that the Lemma still holds, i.e., that the induced map $\Graphs_M^{\fr}(r) \otimes_{ (A\otimes \hat H(G))^r } A^r \cong \Graphs_M(r)\to \OmPA(\FM_M)$ is a quasi-isomorphism.
The proof is analogous to the proof of Theorem \ref{thm:graphs-m-model}, using the result of \cite[Theorem 1, section 9.3]{GHV1976}.

The (collection of the) objects in the upper left corner of the diagram (isomorphic to $\Graphs_M$) inherit a natural $\Graphs_n$-coaction from the $\Graphs_n$-coaction on $\Graphs_M^{\fr}(r)$.
Let us describe that coaction explicitly.
Start with a graph $\Gamma\in \Graphs_M$.
We then take the $\Graphs_n^M$-coaction, producing a product:
\[
  \sum \Gamma' \otimes \gamma \in \Graphs_{M} \otimes \Graphs_{n},
\]
by contracting subgraphs $\gamma$ of $\Gamma$.
Next we coact by $\hat H(G)$ on $\gamma$, producing
\[
  \sum \Gamma' \otimes h \otimes \gamma' \in \Graphs_{M}  \otimes \hat H(G) \otimes \Graphs_{n},
\]
with $h \in \hat H(G)$.
Finally, we use $\sigma$ to map $h$ to $A$ and multiply that factor inside $\Gamma'$, producing:
\[
  \sum \Gamma' p_j^*(\sigma (h)) \otimes \gamma' \in \Graphs_M\otimes \Graphs_n,
\]
which is the desired result of the $\Graphs_n$-coaction.
Let us denote the right $\Graphs_n$-comodule obtained from $\Graphs_M$ via the map $\sigma$ above by $\sigma^*\Graphs_M^{\fr}$.

If we make careful choices, then the trivialization of the frame bundle may be used to define the model $(A\otimes \hat H(G), d)$ for the frame bundle so that the differential is $d=0$.
Furthermore, in this case we also have that $\Graphs_M^{\fr}(r)=\Graphs_M(r) \otimes \hat H(G)^r$, with no piece of the differential going from $\hat H(G)$ to $\Graphs_M$.
In that case we may take our section $\sigma$ to be the trivial map, which is just the projection to the first factor $A$.
The $\Graphs_n$-coaction on $\Graphs_M$ so obtained agrees with the action of \cite{CamposWillwacher2016}.

There is also an alternative viewpoint yielding the same formula.
Our section $s$ of the frame bundle gives rise to a trivialization of the fiberwise little discs operad:
\[
  M \times \FM_n \xrightarrow{s\times \mathit{id}} \Fr_M \times \FM_n \to  \Fr_M \times_G \FM_n\cong \FM_n^M.
\]
Here, the right-hand arrow is the quotient under the $G=\SO(n)$-action, and the whole composition is an isomorphism of operads over $M$.
Using this trivialization we may then pull back the right $\FM_n^M$-multimodule structure on $\FM_M$ to a right $(M \times \FM_n)$-multimodule structure.
This structure restricts trivially to an $\FM_n$-module structure.

Translating this construction into algebraic models, the trivialization is represented by the following composition:
\begin{multline}
  \label{equ:frtriv}
  \Graphs_n^M \to ((A\otimes \hat H(G),d) \otimes \Graphs_n)^{\hat H(G)} \to
  \\
  \to A\otimes \hat H(G) \otimes \Graphs_n
  \xrightarrow{\sigma \otimes \mathit{id}} A \otimes \Graphs_n.
\end{multline}
Here, we used $\Graphs_n$ with the $\hat H(G)$-coaction as a model for $\FM_n$. The first arrow is induced by the $\hat H(G)$-coaction on $\Graphs_n$.
Finally, the (co)trivialization morphism above allows us to (co)restrict the $(A \otimes \Graphs_n)$-coaction on $\Graphs_M$ to a coaction of $\Graphs_n$.
The formula for this coaction is evidently precisely the same as obtained from the previous, alternative derivation.

\subsection{Computation: Changing the frame}

In this subsection, we still assume that $M$ is parallelizable.
We want to study the effect of changing the trivialization of the tangent bundle on our graphical models for $\FM_n$ as $\FM_n$-module.
We shall see that the dependence on the chosen framing is relatively minor.

For the purposes of the following computation, let us assume that we have chosen some reference trivialization.
In this case, our dgca model for the frame bundle can be taken to be the tensor product of dgcas $A\otimes \hat H(G)$.
Moreover, suppose that we have made choices so that our model $A\otimes \Graphs_n$ has no piece of the differential between $\Graphs_n$ and $A$.
Suppose also that the MC element $m\in H(\BG)\otimes \GC_n$ controlling the $\SO(n)$-action on $\FM_n$ has already been brought to the form of Theorem \ref{thm:MCform} by a gauge transformation for simplicity.

Now, suppose that we have chosen some other trivialization of the tangent bundle, which is modeled by a dgca map $\sigma: A\otimes \hat H(G) \to A$ as in \eqref{equ:sigma} above.
Since $\hat H(G)$ is equivalent to $H(G)$ as dgcas, up to changing $\sigma$ by a homotopy, we may assume that it is given as a composition:
\[
  A\otimes \hat H(G) \to A\otimes H(G) \to A.
\]
Such a $\sigma$ is completely determined by providing the images of the generators of $H(G)$, i.e., by the images of the Pontryagin classes and, for even $n$, the Euler class.

Next, let us write down explicitly the composition \eqref{equ:frtriv} describing the trivialization map
\[
  f_\sigma : \Graphs_n^M \to A\otimes \Graphs_n
\]
associated to $\sigma$.
Using the explicit form of the action as encoded by $m$ of Theorem \ref{thm:MCform}, we see that for even dimension $n$ the maps sends a graph $\Gamma\in \Graphs_n$ to
\[
  f_\sigma(\Gamma) = (1+\sigma(E) \tadp\cdot ) \Gamma ,
\]
while for $n$ odd the corresponding map is
\[
  f_\sigma(\Gamma) = (1+\sigma(P_{top}) \theta \cdot ) \Gamma ,
\]
where $\theta$ stand for the $\theta$-graph
\[
  \theta =
  \begin{tikzpicture}[baseline=-.65ex]
    \node[int] (v) at (0,-.5) {};
    \node[int] (w) at (0,.5) {};
    \draw (v) edge (w) edge[bend left] (w) edge[bend right] (w);
  \end{tikzpicture}.
\]
Hence for even $n$, the choice of framing enters only through a choice of class in $H^{n-1}(M)$ represented by the image of the Euler class $\sigma(E)$.
For odd $n$, the framing enters only through a class in $H^{2n-3}(M)$ represented by $\sigma(P_{top})$.
Strikingly, this means in particular that for $n>3$ odd the choice of framing does not enter at all into the real homotopy of $\FM_M$ as a right $\FM_n$-module.

Given the maps $f_\sigma$, let us denote the corresponding $\Graphs_n$-right comodule $f_\sigma^*\Graphs_M$.
(It is the same as $\Graphs_M$ as a Hopf collection, but the operadic coaction is twisted by $f_\sigma$.)
Furthermore, let us note that the construction of $\Graphs_M$ depends on a choice of MC element $z\in \GC_M$, which generally depends on choices made.

Let us make the dependence explicit in the notation and write $\Graphs_M^z$ for the moment.
Then one may in fact check that one has an isomorphism of right $\Graphs_n$ comodules
\[
  f_\sigma^*\Graphs_M^z \cong \Graphs_M^{z'},
\]
where $z'\in \GC_M$ is another Maurer--Cartan element which can be obtained through a natural $(A\otimes \GC_n)$-action on $\GC_M$ as
\[
  z' =
  \begin{cases}
    \exp(\sigma(E) \tadp\cdot) z,        & \text{for $n$ even}; \\
    \exp(\sigma(P_{top}) \theta\cdot) z, & \text{for $n$ odd}.
  \end{cases}
\]

For details, we refer to forthcoming work. In summary, one may say that the choice of framing essentially only affects the coefficient of the tadpole graph in the MC element $z$ if $n$ is even, that it only affects the coefficient of the $\theta$-graph if $n=3$, and that it has no effect for $n\geq 5$ odd.

\todo[inline]{Thomas: Here I thought for quite a while how to make the remark precise (=into a Lemma). This is very naturally done describing the actions of various graphical Lie algebras on the graph complexes. However, I think their description would be combinatorially too lengthy for this paper, and fits nicely into the forthcoming work about the homotopy automorphisms anyway. Hence I settled with the "cheap solution" above.}

\appendix

\section{Explicit form of propagator}\label{sec:explicit_prop}

For completeness, let us give an explicit form of the equivariant propagator, i.e., an equivariant form on the $(n-1)$-sphere extending the (round) volume form whose equivariant differential is 0 for $n$ odd or proportional to the Euler class for $n$ even (see Section~\ref{sec:equiv-graph-models}).
Such a formula has been given within the toric Cartan model in \cite[Appendix A]{KhoroshkinWillwacher2017}.
Here we will alternatively use the (non-toric) Cartan model instead:
\[
  (S(\alg{so}_{n}^*[-2])\otimes \OmPA(S^{n-1}))^G.
\]
A basis for $\alg{so}_{n}^*[-2]$ (dual of antisymmetric matrices) is denoted by symbols $u_{ij}=-u_{ji}$, with $1\leq i\neq j \leq n$.
The map from this model to forms on $\FM_{n}(2) = S^{n-1}$ are defined similarly to the map $\Phi$ from~\cite[Section~4.7]{KhoroshkinWillwacher2017}.

\newcommand{\cE}{{\mathcal E}}
Define the operator
\[
  I := \sum_{i<j}u_{ij} \iota_i \iota_j.
\]
Then define the equivariant volume form to be
\[
  \Omega_{sm}^{C} \coloneqq  \iota_\cE \sum_{0\leq k < n/2} c_k I^k (dx_1\cdots dx_n),
\]
where $\cE=\sum_{j=1}^n x_j \frac{\partial}{\partial x_j}$ is the Euler vector field, $c_{k}$ are combinatorial coefficients to be defined shortly, and it is understood that (only) after performing the contractions one restricts the form to the sphere.

Let us now verify that $\Omega_{sm}^{C}$ satisfies the defining equations, see~\cite[Section~6.4]{KhoroshkinWillwacher2017}.
The well-definedness, the fact that it is a volume form of area $1$ on the sphere, and that it is equivariant under the antipodal map and the action of $G$, are proved the same way as in~\cite[Lemma~A.1]{KhoroshkinWillwacher2017}. The unit volume condition fixes the first coefficient to
\[
  c_0 = \frac 1 {\mathit{vol}(S^{n-1})} .
\]

Finally we must check that its image of $\Omega_{sm}^{C}$ under the differential $d_{u} \coloneqq d- \sum_{i,j} u_{ij}x_i \iota_j$ is proportional to the Euler class (for even $n$) or zero (for odd $n$).
First, note that $[d,\iota_\cE] = L_\cE$ is the Lie derivative with respect to $\cE$, which acts by multiplying a constant form by its degree.
Furthermore $d (I^k(dx_1\cdots dx_n)) = 0$, since the form is constant. Hence we obtain
\begin{align*}
  d \Omega_{sm}^{C}
   & = L_\cE \sum_{0\leq k < n/2} c_k I^k (dx_1\cdots dx_n)
  \\&= \sum_{0\leq k < n/2} (n-2k)c_k I^k (dx_1\cdots dx_n).
\end{align*}
Next, define the differential form
\[
  \eta := \sum_k x_k dx_k.
\]
Denoting by $\eta$ also the operator of multiplication with $\eta$ one computes that:
\begin{align*}
  [I,\eta]       & =-\sum_{i,j} u_{ij} x_i\iota_j
                 & \text{and hence}                                      &  &
  [I^{k+1},\eta] & =-(k+1)\left(\sum_{i,j} u_{ij} x_i\iota_j \right)I^k.
\end{align*}
We hence obtain
\begin{align*}
  \sum_{i,j} u_{ij}x_i \iota_j  \Omega_{sm}^{C}
   & =
  -\iota_\cE
  \sum_{0\leq k < n/2} c_k
  \left(\sum_{i,j} u_{ij} x_i\iota_j \right)I^k (dx_1\cdots dx_n)
  \\&=
  \iota_\cE
  \sum_{0\leq k < n/2} \frac{c_k}{k+1}
  [I^{k+1},\eta](dx_1\cdots dx_n)
  \\&=
  \iota_\cE
  \sum_{0\leq k < n/2} \frac{c_k}{k+1}
  (I^{k+1}\eta-\eta I^{k+1})(dx_1\cdots dx_n).
\end{align*}
By degree reasons we have $\eta (dx_1\cdots dx_n)=0$. Hence the expression above becomes
\[
  -\iota_\cE
  \sum_{0\leq k < n/2} \frac{c_k}{k+1}
  \eta I^{k+1}(dx_1\cdots dx_n)
  =
  -\sum_{0\leq k < n/2} \frac{c_k}{k+1}
  ([\iota_\cE,\eta] -\eta \iota_\cE) I^{k+1}(dx_1\cdots dx_n)
\]

Furthermore, we note $[\iota_\cE,\eta]=\sum_kx_k^2=1$ when restricted to the sphere $S^{n-1}$, and that $\eta=\frac 1 2 d \sum_kx_k^2=0$ when restricted to the sphere. We hence find, on the sphere
\[
  \sum_{i,j} u_{ij}x_i \iota_j  \Omega_{sm}^{C}
  =
  -\sum_{0\leq k < n/2} \frac{c_k}{k+1}
  I^{k+1}(dx_1\cdots dx_n).
\]
Putting the above computations together we see that, again on the sphere
\begin{equation}\label{equ:duOmega}
  \begin{aligned}
    d_u \Omega_{sm}^{C}
     & =
    \sum_{0\leq k < n/2}
    \left(
    (n-2k)c_k I^k
    +
    \frac{c_k}{k+1} I^{k+1}
    \right)
    (dx_1\cdots dx_n)
    \\&=
    \sum_{1\leq k < n/2}
    \Bigl(
    (n-2k)c_k
    +
    \frac{c_{k-1}}{k}
    \Bigr)
    I^k(dx_1\cdots dx_n) \mathbin{+}
    \\ & \hspace{2cm}
    \mathbin{+}
    \begin{cases}
      \frac{c_{n/2-1}}{n/2} I^{n/2}(dx_1\cdots dx_n)
        & \text{for $n$ even}
      \\
      0 & \text{for $n$ odd}
    \end{cases}.
  \end{aligned}
\end{equation}
where we use that $I^{(n-1)/2+1}dx_1\cdots dx_n=0$ for odd $n$ and $dx_1\cdots dx_n=0$ on the sphere by degree reasons. The vanishing of the sum over $k$ yields a recursive formula for the $c_k$ which can be solved to
\[
  c_k
  =
  \frac{(-1)^k}
  {k! (n-2k)(n-2k+2)\cdots (n-2)} c_0,
\]
with $c_0=\frac{1}{\mathit{vol}(S^{n-1})}=\frac{\Gamma(n/2)}{2\pi^{n/2}}$.
For odd $n$ we then find from \eqref{equ:duOmega} that $d_u \Omega_{sm}^{C}=0$ as desired.
For even $n=2m$ denote
\[
  \Pf(u):= \frac{1}{2^m m!} \sum_{\sigma\in S_n}(-1)^\sigma
  u_{\sigma(1)\sigma(2)}\cdots u_{\sigma(n-1)\sigma(n)}.
\]
We have
$$
I^{n/2}(dx_1\cdots dx_n)=(-1)^m m! \Pf(u),
$$
so that equation \eqref{equ:duOmega} simplifies to
\begin{align*}
  d_u \Omega_{sm}^{C} & =
  \frac{c_{m-1}}{m} (-1)^m m! \Pf(u)
  \\&=
  \frac{-1}
  {m!2^{m-1}(m-1)!}
  \frac{(m-1)!}{2\pi^m}
  m!
  \Pf(u)
  \\&=
  -\frac{1}{(2\pi)^m}
  \Pf(u)=:-E,
\end{align*}
where we define the Euler class as $E=\frac{\Pf(u)}{(2\pi)^m}$.

\section{Homotopy commutativity of a diagram}
\label{sec:homot-comm-diagr}

The purpose of this section is to show the following Proposition.
\begin{prop}\label{prop:bgraphsdiag}
  The following diagram of homotopy Hopf cooperads under $H(\BG)$ is homotopy commutative:
  \begin{equation}
    \label{equ:bgraphsdiag}
    \begin{tikzcd}[cramped, sep=small]
      \BGraphs^m_n \ar{r} \ar{d} & \Tot ( \OmPA(G^\bullet\times \FM_n)) \ar{d}\\
      (H(\BG)\otimes \hat H(G)\otimes H(\BG) \otimes \Graphs_n,d) \ar{r}
      & \Tot ( \OmPA(G^\bullet\times WG \times G^\bullet\times \FM_n))
    \end{tikzcd},
  \end{equation}
  where we consider the diagonal of the bisimplicial object on the bottom-right corner, and where the left vertical arrow is induced through the coaction (according to the MC element $m$)
  \[
    \Graphs_n\to \hat H(G)\otimes \Graphs_n.
  \]
\end{prop}

Pulling back along the map $A\to \OmPA(M)$, we obtain the following corollary, which is a key part of Section~\ref{sec:zigzag} (see diagram~\eqref{equ:AGraphshoCD}):
\begin{cor}\label{cor:bgraphsdiag}
  The following diagram of homotopy operads under $A\simeq \OmPA(M)$ is homotopy commutative:
  \beq{equ:bgraphsdiag2}
  \begin{tikzcd}
    \Graphs_n^M \ar{r} \ar{d} & \OmPA(\FM_n^M) \ar{d}\\
    (A \otimes \hat H(G)\otimes H(\BG) \otimes \Graphs_n,d) \ar{r}
    & \Tot ( \OmPA(\Fr_M^W \times G^\bullet\times \FM_n))\, .
  \end{tikzcd}
  \eeq
\end{cor}

We will in fact show that one has a homotopy, in the naive sense that one has a map:
\[
  \BGraphs^m_n \to \Tot ( \OmPA(G^\bullet\times WG \times G^\bullet\times \FM_n))[t,dt]
\]
compatible with the homotopy cooperadic structure, whose restriction to $t=0$ agrees with the upper rim of the above diagram, and whose restriction to $t=1$ agrees with the lower rim.

To show the statement we need several auxiliary constructions and lemmas that we are going to detail in the next sections.

\subsection{Two models for $H(\BG)$}
\label{app:cartan-model}

\newcommand{\Car}{\mathrm{Car}}
In this section, we study two models for $\BG$.

First, in Equation~\eqref{equ:bgraphsdiag} above, the dgca $(H(\BG)\otimes \hat H(G)\otimes H(\BG),d)$ appears.
Here, $\hat H(G)$ is the canonical (Koszul) resolution of the coalgebra $H(G)$, i.e., it is a cofree coassociative coalgebra cogenerated by the reduced cohomology $\tilde H(\BG)[-1]$ (see also Section~\ref{sec:equiv-graph-models}).
Put differently, elements of $\hat H(G)$ are words in a basis of monomials in the Pontryagin and Euler classes.
A typical element of $\hat H(G)$ is, for example,
\[
  (P_4 P_8^2)(P_{16}) (P_4^3) \in \hat{H}(G).
\]
The product on such words is the shuffle product, and the differential on $\hat H(G)$ merges two adjacent monomials.
Within the dgca $H(\BG)\otimes \hat H(G)\otimes H(\BG)$, there is an additional piece of the differential, which takes the first (respectively last) monomial in the word and identifies it with an element of the left (respectively right) factor $H(\BG)$.
In other words, it is the two-sided cobar construction $\Omega(H(\BG), \tilde{H}(\BG)[-1], H(\BG))$, and as such it is a resolution of $H(\BG)$.

The second model of $H(\BG)$ we consider is as follows.
We consider a version of the Cartan model
\[
  \Car :=  (\R[u_{ij}, v_{ij}, \tilde u_{ij} \mid 1 \leq i, j \leq n])^G.
\]
The generators $u_{ij}=-u_{ji} \in \mathfrak{so}_{n}^{*}[-2]$ and $\tilde u_{ij}=-\tilde u_{ji}$ (of a different copy of $\mathfrak{so}_{n}^{*}[-2]$) have degree $+2$, while the generators $v_{ij}=-v_{ji}$ have degree $+1$.
The differential is defined such that $dv_{ij}=u_{ij}-\tilde u_{ij}$.

In what follows, for $f\in H(\BG)$ a polynomial in Euler and Pontryagin classes, we denote (slightly abusively) the image polynomial in variables $u_{ij}$ by
\[
  f(\dots, u_{ij}, \dots),
\]
and similarly for $\tilde{u}_{ij}$ and $v_{ij}$.

There is a direct map
\begin{equation}
  \label{eq:cmp-two-models}
  \Psi : (H(\BG)\otimes \hat H(G)\otimes H(\BG),d) \to \Car
\end{equation}
defined as follows:
\begin{itemize}
  \item On the left (resp.\ right) $H(\BG)\cong \R[a_{ij}]^G$, $\Psi$ is defined by replacing $a_{ij}$ by $u_{ij}$ (resp.\ $\tilde u_{ij}$).
        In other words, given an element $f \in H(\BG)$, we consider either $f(\dots,u_{ij},\dots)$ or $f(\dots,\tilde{u}_{ij},\dots)$.
  \item Let us formally denote, for $t\in \R$,
        \begin{equation}
          \label{eq:xijt}
          x_{ij,t} \coloneqq (1-t) u_{ij} +t\tilde u_{ij} - dt v_{ij}.
        \end{equation}
        The map $\hat H(G)\to \Car$ sends a word
        \[
          f_1\cdots f_k
        \]
        to the $k$-fold iterated integral
        \[
          \Psi(1 \otimes f_1 \dots f_k \otimes 1)
          \coloneqq
          \iiint_{\;\mathrlap{0\leq t_1\leq t_2 \leq \cdots \leq t_k\leq 1}}
          f_1(\dots, x_{ij,t_1},\dots)
          \cdots
          f_k(\dots, x_{ij,t_k},\dots).
        \]
\end{itemize}

\begin{lemma}
  The map $\Psi$ of Equation~\eqref{eq:cmp-two-models} is a quasi-isomorphism of dgcas $(H(\BG)\otimes \hat H(G)\otimes H(\BG),d) \to \Car$.
\end{lemma}
\begin{proof}
  First, let us check that the map intertwines the commutative products.
  This is clear on the factors $H(\BG)$.
  On the factor $\hat H(G)$ it follows from the usual shuffle formula for iterated integrals.

  Second, we check that the map commutes with the differentials.
  This follows easily from Stokes' Theorem, and the fact that the $x_{ij,t}$ are closed under the combined differential (the differential on the complex plus the de Rham differential in $t$).

  Finally, we check the quasi-isomorphism property.
  Both inclusions of $H(\BG)$ as the first factor in both the domain and the target are quasi-isomorphisms.
  The result thus follows by the 2-out-of-3 property of quasi-isomorphisms.
\end{proof}

\subsection{Variants of graph complexes and graph cooperads}

We will now consider graph complexes $\biGC_n$ and $\biGr_n$.
The idea is that we are going to build a cylinder object for $\Graphs_n$.

These complexes are defined similarly to $\GC_n$ and $\Graphs_n$, except that we distinguish three types of edges.
We call these type $u$-edges, $\tilde u$-edges and $v$-edges, marked by an appropriate letter in drawings.
We impose the differential on $\biGr_n$
\[
  d\,
  \bigl(
  \begin{tikzpicture}[baseline=-1ex]
      \draw (0,0) edge node[above]{$\scriptstyle v$} (1,0);
    \end{tikzpicture}
  \bigr)
  =
  \begin{tikzpicture}[baseline=-1ex]
    \draw (0,0) edge node[above]{$\scriptstyle u$} (1,0);
  \end{tikzpicture}
  \,
  -
  \,
  \begin{tikzpicture}[baseline=-1ex]
    \draw (0,0) edge node[above]{$\scriptstyle \tilde u$} (1,0);
  \end{tikzpicture}
\]
In particular the $v$-edges have degree $n-2$, while the $u$-edges and $\tilde{u}$-edges have degree $n-1$ in $\biGr_n$.
The other summands of the differential contract the $u$- and $\tilde u$-type edges, just like in $\Graphs_{n}$.
We obtain natural maps:
\beq{equ:bigrmaps}
\begin{tikzcd}
  \Graphs_n \ar[shift right,swap]{r}{\phi_1} \ar[shift left]{r}{\phi_0} & \biGr_n \ar{r} & \Graphs_n,
\end{tikzcd}
\eeq
where the map $\phi_0$ on the left send a graph to the same graph with all edges marked by $u$, the map $\phi_1$ marks all edges by $\tilde u$, and the right-hand map identifies (forgets) colors $u$ and $\tilde u$ and sends $v$ to zero.
All these maps are quasi-isomorphisms.

For $\biGC_n$, we have the corresponding dual differential and grading conventions.
Dually, we have maps defined likewise on the dual complexes of connected graphs with only internal vertices:
\begin{equation}
  \label{eq:gc-bigc}
  \GC_n \to \biGC_n \rightrightarrows \GC_n.
\end{equation}

Recall the Cartan model $\Car = \bigl( \R[u_{ij}, v_{ij}, \tilde{u}_{ij}] \bigr)^{G}$ of $H(\BG)$ (see Section~\ref{app:cartan-model}).
We now define a differential on $\Car \otimes \biGr_{n}$ such that there is a natural map of homotopy cooperads over $\Car$,
\begin{equation*}
  \BGBbiGr_{n} \coloneqq (\Car \otimes \biGr_n, d)
  \to
  \Tot \OmPA(G^\bullet \times WG \times G^\bullet \times \FM_n)^{W_r},
\end{equation*}
given as follows:
\begin{itemize}
  \item The $u$-edges are sent to the corresponding equivariant propagators in the $u_{ij}$.
  \item The $\tilde u$-edges are sent to the propagators in the $\tilde u_{ij}$.
  \item The $v$-edges are sent to interpolating forms between the two equivariant propagators, defined similarly to $x_{ij,t}$ from Equation~\eqref{eq:xijt}.
  \item On $\Car$, the map is defined similarly to the map from the non-toric Cartan model from Appendix~\ref{sec:explicit_prop}.
\end{itemize}
The differential on $\BGBbiGr_{n}$ uses the MC element $m \in H(\BG) \hotimes \GC_{n}$ from Equation~\eqref{eq:mc-elt-m}.
We first map it to $m^{\mathrm{bi}} \in H(\BG) \hotimes \biGC_{n}$ using the map $\GC_{n} \to \biGC_{n}$ of Equation~\eqref{eq:gc-bigc}.
Then the differential of $\Gamma \in \biGr_{n}$ is $1 \otimes 1 \otimes m^{\mathrm{bi}} \cdot \Gamma \in H(\BG) \otimes \hat{H}(G) \otimes H(\BG) \otimes \biGr_{n}$.

\subsection{Construction}
\label{sec:construction}
We now describe a general construction that we will apply in the next section.
Suppose that $\alg g$ and $\alg h$ are dg Lie algebras, acting on modules $U$ and $V$, respectively.
Suppose that we have maps of dg Lie algebras $f:\alg h \to \alg g$ and of modules $F:U\to V$. Concretely, for $u\in U$ and $h\in\alg h$:
\[
  h\cdot F(u) = F(f(h) \cdot u).
\]
Suppose that $\mu\in \alg g$ and $m\in \alg h$ are Maurer--Cartan elements.
Finally, suppose that we have a gauge transformation:
\[
  f(m) \simeq \mu.
\]
Such a gauge transformation may be integrated (provided suitable (pro-)nilpotence properties) to a group-like element $A\in \U\alg g$ satisfying
\[
  A^{-1} f(m) A = \mu.
\]
Under these conditions we may build a map of the twisted dg vector spaces (provided again suitable nilpotence conditions guaranteeing convergence)
\begin{gather*}
  F^{m} : U^\mu \to V^{m} \\
  F^{m}(u) = F( A\cdot u).
\end{gather*}
It is an elementary exercise to verify that this map indeed intertwines the differentials.

Let us remark on a special case of this construction that will be used later.
Suppose that in fact our Lie algebras and modules are defined over the ground ring $\R[t,dt]$, and more specifically, assume that
\begin{align*}
  \alg g & = \alg g'[t,dt]
         &
  \alg h & = \alg h'[t,dt]
  \\
  U      & =U'[t,dt]       & V & =V'[t,dt],
\end{align*}
where the actions are extended from actions of $\alg g'$ on $U'$ and $\alg h'$ on $V'$.
We assume that our MC element $m$ above has no $t$-dependence, i.e., that $m\in \alg h'$.
Then $\tilde m = \tilde m_t+ h_t \cdot dt := f(m)\in \alg h$ encodes a family of gauge equivalent MC elements $\tilde m_t\in \alg g$.
Let us choose for the MC element entering the above construction $\mu:= \tilde m_0\in \alg g'\subset \alg g$.
Then indeed $\mu$ and $f(m)$ are gauge equivalent MC elements.
The gauge equivalence is encoded by the MC element
\[
  f(m(t\mapsto st)) \in \alg g[s,ds]
\]
obtained by formally replacing $t$ by $st$ in $\tilde m=f(m)$.
In this case, one can check that the element $A\in \U \alg g$ above is (as function of $t$) the gauge flow encoded by $\tilde m$ up to time $t$. In other words, making explicit the time dependence, $A_t$ is obtained by solving the ODE
\[
  \dot{A}_{t} = h_{t} A_{t}.
\]
Eventually, our construction then produces an $\R[t,dt]$-linear map
\[
  F = U^\mu \to V^m= (V')^m[t,dt],
\]
This map can be seen as an explicit homotopy for the family of maps
\begin{gather*}
  F_t^m : (U')^\mu \to (V')^m \\
  F_t^m(u) = F_t(A_t u).
\end{gather*}

In this construction we use only natural operations. Hence, if there is additional structure (dgcas, cooperads) on our spaces $U$ and $V$ and this structure is preserved by the Lie actions, then our homotopy also is compatible with the structure given.

\subsection{Two maps, and the proof of Proposition \ref{prop:bgraphsdiag}}
We describe next two maps of Hopf cooperads under $H(\BG)$
\[
  F_0,F_1: \BGraphs_n^m \rightrightarrows \BGBbiGr_n = (\Car\otimes \biGr_n,d).
\]
The first sends a graph $\Gamma\in \Graphs_n$ to the same graph with all edges colored $u$.
In other words $F_0$ agrees with the $H(\BG)$-linear extension of \eqref{equ:bigrmaps}.
The second map $F_1$ is the composition
\[
  \Gamma\mapsto F_1(\Gamma) = \phi_1(A\Gamma),
\]
where $\phi_1$ is, up to $H(\BG)$-linear extension, the map from \eqref{equ:bigrmaps} coloring all edges by $\tilde u$.
The element $A\in \Car \otimes \hat{\U}\GC_n$ is a group-like element obtained as the parallel transport from $t=0$ to $t=1$ of the MC element (and in particular flat connection on the interval):
\[
  m(\dots, x_{ij,t},\dots) \in \Car\otimes \GC_n [t,dt].
\]
More concretely, the characteristic property of the element $A$ is that
\[
  A^{-1}m_1A  =  m_0,
\]
where $m_0=m(\dots, u_{ij},\dots)$ and $m_1=m(\dots, \tilde u_{ij},\dots )$ are the MC elements at the endpoints of the interval.
The construction of the maps $F_0,F_1$ fits exactly the construction of the previous subsection.
In particular they intertwine the differentials properly.
Furthermore, as we saw in the previous subsection $F_0$ and $F_1$ are homotopic.
We shall mark this result in the following lemma:
\begin{lemma}\label{lem:outerrim}
  The two maps
  \[
    F_0,F_1:\BGraphs_n^m \rightrightarrows \BGBbiGr_n
  \]
  are homotopic.
\end{lemma}

Finally, we note the following result, which follows by explicit computation.
\begin{lemma}
  The upper and lower compositions
  \[
    \begin{tikzcd}
      \BGraphs_n^m \ar[shift left]{r}{F_0} \ar[shift right,swap]{r}{F_1} & \BGBbiGr_n \ar{r} \ar{r} & \Tot ( \OmPA(G^\bullet\times WG \times G^\bullet\times \FM_n))
    \end{tikzcd}
  \]
  agree with the upper and lower composition in the diagram \eqref{equ:bgraphsdiag}.
\end{lemma}

Hence, Proposition \ref{prop:bgraphsdiag} follows immediately from the preceding two lemmas and the construction of Section~\ref{sec:construction}.

\section{Hopf formality}\label{sec:hopf-formality}

In this section, we prove a Hopf formality result for compact Lie groups that could be found in \cite[Proposition 4.5]{KhoroshkinWillwacher2017}.

\begin{prop}\label{prop:hopf-formal}
  Let $G$ be a compact connected Lie group and let $A_G = W\OmPA(G)$ be the complete Hopf algebra from Definition~\ref{def:a-g}.
  There exists a quasi-isomorphism of complete Hopf algebras:
  \begin{equation}
    \upsilon_G : H(G) \to A_G .
  \end{equation}
  It follows that $A_G$ is formal.
\end{prop}

\begin{proof}
  Let $\oW\OmPA(G) \subset W\OmPA(G)$ be the set of primitive elements, so that $A_G$ is the bar construction on $\oW\OmPA(G)$.
  Moreover, let us denote by $V \subset H(G)$ the space of primitive elements (concentrated in odd degrees), and note that since $G$ is connected, $H(G)$ is generated by $V$.
  Since $H(G)$ is free and $A_G$ is quasi-cofree, any map $F : H(G) \to A_G$ is determined by its restriction on generators composed with the projection on cogenerators,
  \begin{equation}
    f : V \to \oW\OmPA(G).
  \end{equation}
  Moreover, the condition that $F$ is a morphism of complete (dg) Hopf algebras is equivalent to the Maurer--Cartan equation:
  \begin{equation}\label{eq:mc-hopf-fmlt}
    \forall v \in V, \; d_{\oW\OmPA(G)} f(v) + \sum_{(v)} F(v') F(v'') = 0,
  \end{equation}
  where we use the Sweedler notation $\Delta(v) = \sum_{(v)} v' \otimes v''$.
  Since $V$ consists of primitive elements, the second term actually vanishes, so all we need to do is find a map $f : V \to \oW\OmPA(G)$ which lands in the space of closed elements.

  Let us build such a map $f$ inductively, using the filtration $\folF_\bullet$ by the number of vertices in the linear trees underlying $\oW\OmPA(G)$.
  For this filtration, just note that the associated graded of $\Hom(V, \oW(\OmPA(G)))$ has cohomology $\Hom(V, \BarOp^c(H(G)))$.

  We would thus like to build maps $f_p : H(G) \to \folF_1 / \folF_{p+1}$ that satisfy condition~\eqref{eq:mc-hopf-fmlt}.
  The case $f_0$ is trivial.
  For $f_1 : H(G) \to \folF_1 / \folF_2$, we just choose any map $V \to \OmPA(G)$ which takes a generator to a closed representative; condition~\eqref{eq:mc-hopf-fmlt} is clearly satisfied.

  Now, let us assume that we have built $f_p : H(G) \to \folF_1 / \folF_{p+1}$.
  To extend such a map to $f_{p+1}$, we must find closed forms in $\OmPA([0,1])^{\otimes p} \otimes \OmPA(G^{p+1})$ with the prescribed boundary values.

  To illustrate the question, let us consider $p = 1$.
  To extend $f_1$ to $f_2$, for a generator $v \in H^k(G)$, there are three trees to consider:
  \begin{align*}
    \begin{tikzpicture}[baseline=(b)]
      \node {} [ourTree] child {
          node[int] (b) {}
          child {node{}}
        };
    \end{tikzpicture}
     & \xmapsto{f_2(v)}
    \omega \in \OmPA(G),
     &
    \begin{tikzpicture}[baseline=(b)]
      \node {} [ourTree] child {
          node[int] (b) {}
          child {
              node[int]{}
              child {node{}}
            }
        };
    \end{tikzpicture}
     & \xmapsto{f_2(v)}
    \alpha \otimes p \in \OmPA(G \times G) \otimes \OmPA([0,1]),
    \\
     &                  &
    \begin{tikzpicture}[baseline=(b)]
      \node {} [ourTree] child {
          node[int] (b) {}
          child {
              node[int]{}
              edge from parent[dashed]
              child {node{} edge from parent[solid]}
            }
        };
    \end{tikzpicture}
     & \xmapsto{f_2(v)}
    \nu_1 \otimes \nu_2 \in \OmPA(G) \otimes \OmPA(G)
  \end{align*}
  Since $f_2$ extends $f_1$, we have that $\omega = f_1(v)$.
  Moreover, the fact that $F_2$ is a map of coalgebras determines $\nu_1 \otimes \nu_2 = (f_1 \otimes f_1)(\Delta(v))$.
  So, extending $f_1$ to $f_2$ boils down to find $\alpha \otimes p$ such that $p(1) \alpha = \nu_1 \times \nu_2$ and $p(0) \alpha = m^* \omega$, where $m : G \times G \to G$ is the product.
  The obstruction to find such a form lies in $H^{k+1-1}(G^2) = H^k(G^2)$, and solutions are parametrized (up to exact terms) by $H^{k-1}(G^2)$.

  Let us now outline the general case.
  For a generator $v \in V^k$ of degree $k$, the obstruction for finding such elements lives in
  \[H^{k+1-p}(G^{p+1}) \subset H^{k}(G^{p+1} \times \partial [0,1]^p).\]
  Moreover, if the obstruction vanishes, the possibles choices differ by closed forms $u$ vanishing on the boundary of the cube, and as such are parametrized (up to exact terms) by $H^{k-p}(G^{p+1}) \subset H^{k-p+1}(G^{p+1} \times [0,1], G^{p+1} \times \partial [0,1])$.

  We may identify $H(G^{p+1})$ as a subspace of the cobar construction $\BarOp^c(H(G))$.
  Our obstruction is, in fact, a closed element there as it is obtained by capping a coboundary with the fundamental class of $\partial [0,1]^p$.
  By exercising our choices by adding $u$ in the previous induction step we add the cobar differential $d_{\BarOp^c} u$ to our obstruction (as we can add a coboundary of a form of the type $\beta \times \mathrm{vol}_{\partial [0,1]^p}$).
  Hence the obstruction lives in $H^{k}(\BarOp^c H(G))$, taking into account a degree shift in the definition of the cobar construction.
  But here $k$ is odd (all generators of $H(G)$ are in odd degrees), while $\BarOp^c H(G)$ computes $H(BG)$ which is concentrated in even degrees.
  The obstruction must therefore vanish.
\end{proof}

Moreover, this quasi-isomorphism can be extended as follows. Compare with \cite[Proposition 4.11]{KhoroshkinWillwacher2017}, where the following notation differs: our $K_G$~\eqref{eq:koszul-complex} is denoted $\tilde{K}$, and our $\modo_{A_G}(*)$~\eqref{eq:mod-a-g} is denoted $K_G$.

\begin{prop}\label{prop:hopf-extension}
  The quasi-isomorphism $\upsilon_G : H(G) \to A_G$ from Proposition~\ref{prop:hopf-formal} can be extended into a quasi-isomorphism:
  \begin{equation}
    \Upsilon_G : K_G \xrightarrow{\sim} \modo_{A_G}(*),
  \end{equation}
  where $K_G$ is the Koszul complex~\eqref{eq:koszul-complex},
  and the construction $\modo_{A_G}$ is defined in \eqref{eq:mod-a-g}.

  The map $\Upsilon_G$ is a quasi-isomorphism of dgcas (and therefore gives $\modo_{A_G}(*)$ an $H(\BG)$-module structure).
  Moreover, $\Upsilon_G$ is a morphism of $A_G$-comodules, using $\upsilon_G$ from Proposition~\ref{prop:hopf-formal} to view $K_G$ as an $A_G$-comodule, i.e., the following diagram commutes (where dashed arrows represent coactions):
  \begin{equation}
    \begin{tikzcd}
      H(G) \ar[r, "\upsilon_G"] \ar[d, dashed, "\text{coact}"] & A_G \ar[d, dashed, "\text{coact}"] \\
      K_G \ar[r, "\Upsilon_G"] & \modo_{A_G}(*).
    \end{tikzcd}
  \end{equation}
\end{prop}

\begin{proof}
  This follows from similar obstruction-theoretic arguments as in the proof of Proposition~\ref{prop:hopf-formal}.
  Indeed, $K_G$ is free as a dgca and $\modo_{A_G}(*)$ is defined as a $W$-construction under the $A_G$-comodule structure.
  The idea of the proof is to build a map $f : K_G \to \modo_{A_G}(*)$ by induction on the number of vertices in the linear trees underlying $K_G$, just as before.
\end{proof}

\makeatletter
\providecommand\@dotsep{5}
\makeatother
\listoftodos\relax

\printbibliography

\end{document}

%% file: preamble.tex
\usepackage[T1]{fontenc}
\usepackage[utf8]{inputenc}
\usepackage[english]{babel}
\usepackage{lmodern}
\usepackage{exscale}
\usepackage[babel]{microtype}
\usepackage{amsmath, amssymb, mathtools, mathrsfs}

\usepackage{csquotes}
\usepackage[style=alphabetic, sorting=nyt, maxnames=10, maxalphanames=5]{biblatex}
\usepackage{tikz}
\usetikzlibrary{matrix,shapes,arrows,calc,3d,decorations,decorations.pathmorphing,through,cd}
\usepackage[disable]{todonotes}

\newcommand{\ricardoline}[1]{\todo[inline,color=red!40]{#1}}
\newcommand{\najib}[1]{\todo[color=blue!40]{#1}}
\newcommand{\najibline}[1]{\todo[inline,color=blue!40]{#1}}

\usepackage[numbered]{bookmark}
\usepackage{hyperref}
\hypersetup{
  colorlinks,
  urlcolor = {blue!80!black},
  linkcolor = {red!70!black},
  citecolor = {green!50!black}
}

\theoremstyle{plain}
  \newtheorem{thm}{Theorem}
  
  \newtheorem{prop}[thm]{Proposition}
  
  \newtheorem{cor}[thm]{Corollary}
  \newtheorem{lemma}[thm]{Lemma}
\theoremstyle{definition}
  \newtheorem{defi}[thm]{Definition}
  \newtheorem{convention}[thm]{Convention}
\theoremstyle{remark}
  \newtheorem{ex}[thm]{Example}
  \newtheorem{rem}[thm]{Remark}

%% file: math.tex
\tikzset{ext/.style={circle, draw,inner sep=1pt},int/.style={circle,draw,fill,inner sep=1pt},nil/.style={inner sep=1pt}}
\tikzset{exte/.style={circle, draw,inner sep=3pt},inte/.style={circle,draw,fill,inner sep=3pt}}
\tikzset{diagram/.style={matrix of math nodes, row sep=3em, column sep=2.5em, text height=1.5ex, text depth=0.25ex}}
\tikzset{diagram2/.style={matrix of math nodes, row sep=0.5em, column sep=0.5em, text height=1.5ex, text depth=0.25ex}}

\newcommand{\alg}[1]{\mathfrak{{#1}}}

\newcommand{\folF}{\mathcal{F}}

\newcommand{\R}{{\mathbb{R}}}
\newcommand{\Q}{\mathbb{Q}}
\newcommand{\Z}{{\mathbb{Z}}}
\newcommand{\U}{{\mathcal{U}}}

\newcommand{\Graphs}{{\mathsf{Graphs}}}

\newcommand{\BGraphs}{\mathsf{BGraphs}}
\newcommand{\pdu}{}

\newcommand{\sslash}{/\mkern-5mu/}
\newcommand{\qiso}{\xrightarrow{\sim}}
\newcommand{\qisor}{\xleftarrow{\sim}}

\newcommand{\Gra}{{\mathsf{Gra}}}

\newcommand{\biGC}{\GC^{\mathrm{bi}}}
\newcommand{\biGr}{\Graphs^{\mathrm{bi}}}
\newcommand{\BGBbiGr}{\mathsf{BGBGraphs}^{\mathrm{bi}}}

\newcommand{\op}{\mathcal}

\newcommand{\Lie}{\mathsf{Lie}}

\newcommand{\SO}{\mathrm{SO}}

\newcommand{\fT}{\mathsf{Tree}}
\newcommand{\fTm}{\mathsf{Tree}_{*}}
\newcommand{\catC}{\mathcal C}
\newcommand{\dgca}{\mathrm{Dgca}}

\newcommand{\FM}{\mathsf{FM}}
\newcommand{\FFM}{\mathsf{FM}^{\mathrm{fr}}}

\newcommand{\bpm}{\begin{pmatrix}}
\newcommand{\epm}{\end{pmatrix}}

\newcommand{\GC}{\mathrm{GC}}

\newcommand{\BG}{\mathrm{BG}}
\newcommand{\BO}{\mathrm{BO}}

\newcommand{\hotimes}{\mathbin{\hat\otimes}}

\DeclareMathOperator{\Hom}{Hom}
\DeclareMathOperator{\Emb}{Emb}
\DeclareMathOperator{\Map}{Map}

\DeclareMathOperator{\Discs}{Discs}
\DeclareMathOperator{\Conf}{Conf}

\DeclareMathOperator{\id}{id}
\DeclareMathOperator{\Tot}{Tot}

\DeclareMathOperator{\Pf}{Pf}
\DeclareMathOperator{\modo}{mod}

\newcommand{\stG}{\Graphs}

\newcommand{\beq}[1]{\begin{equation}\label{#1}
}
\newcommand{\eeq}{\end{equation}}
\newcommand{\Fr}{\mathrm{Fr}}
\newcommand{\fr}{\mathrm{fr}}

\newcommand{\Top}{\mathsf{Top}}

\DeclareMathOperator{\OmPA}{\Omega_{\mathrm{PA}}}
\newcommand{\oW}{\mathring{W}}
\DeclareMathOperator{\BarOp}{Bar}

\newcommand{\tadp}
{{
\begin{tikzpicture}[baseline=-.55ex,scale=.7, every loop/.style={}]
 \node[circle,draw,fill,inner sep=.5pt] (a) at (0,0) {};
 \draw (a) edge[loop] (a);
\end{tikzpicture}}}

\tikzset{ourTree/.style = {grow' = up, level distance = 5mm}}